\documentclass[11pt]{article}

\usepackage{amsmath,amsfonts,amsthm,amsrefs}

\usepackage{graphicx,color,pict2e}

\definecolor{refkey}{rgb}{0.1, 0.2, 1}  
\definecolor{labelkey}{rgb}{0.1,1,0.2}

\newtheorem{theorem}{Theorem}[section]
\newtheorem{lemma}[theorem]{Lemma}
\newtheorem{remark}[theorem]{Remark}
\newtheorem{definition}[theorem]{Definition}

\newcommand{\abs}[1]{\left| #1 \right|}

\def\implies{\Longrightarrow}
\def\bel{\begin{equation}\label}
\def\eeq{\end{equation}}
\def\bega{\begin{array}}
\def\enda{\end{array}}
\def\v{\vskip 1em}
\def\ve{\varepsilon}
\def\R{{\mathbb R}}

\def\O{\mathcal{O}}
\def\P{\mathcal{P}}
\def\dx{\Delta x}
\def\eps{\epsilon}
\def\C{\mathcal{C}}
\def\Hat{\widehat}

\def\cP{\mathcal{P}}
\def\cQ{\mathcal{Q}}
\def\cA{\mathcal{A}}

\title{Traveling Waves for Nonlocal Models of Traffic Flow}
\author{Johanna Ridder and Wen Shen \\ Mathematics Department, Pennsylvania State University \\ Emails: jur436@psu.edu and wxs27@psu.edu}

\begin{document}

\maketitle

\begin{abstract}
We consider several non-local models for traffic flow, 
including both microscopic ODE models and macroscopic PDE models. 
The ODE models describe the movement of individual cars,
where each driver adjusts the speed according to the road condition over 
an interval in the front of the car. 
These models are known as the FtLs (Follow-the-Leaders) models. 
The corresponding PDE models, describing the evolution for the density of cars, 
are conservation laws with non-local flux functions. 
For both types of models, we study stationary traveling wave profiles and
stationary discrete traveling wave profiles.
(See definitions~\ref{def:TW} and~\ref{def:DTW}, respectively.)
We derive delay differential equations satisfied by the profiles for the FtLs models,
and delay integro-differential equations for the traveling waves of the 
nonlocal PDE models. 
The existence and uniqueness (up to horizontal shifts) of the stationary traveling wave
profiles are established. 
Furthermore,  we show that 
the traveling wave profiles are time asymptotic limits for the corresponding Cauchy problems,
under mild assumptions on the smooth initial condition. 
\end{abstract}



\section{Introduction}
\setcounter{equation}{0}

We consider the Cauchy problem for two conservation laws with nonlocal flux describing traffic flow, 
\begin{equation}\label{eq:claw1}
\rho_t(t,x) + \left[ \rho(t,x) \cdot v\left(\int_x^{x+h} \rho(t,y) w(y-x) \; dy\right)\right]_x =0,
\end{equation}
and
\begin{equation}\label{eq:claw2}
\rho_t(t,x) + \left[ \rho(t,x) \cdot \left(\int_x^{x+h}v (\rho(t,y)) w(y-x) \; dy\right)\right]_x =0,
\end{equation}
where $x,t\in\mathbb{R}$ and $t\ge 0$. 
In both models,  $\rho$ is the density function of cars, and 
$h\in\mathbb{R}$ satisfies $h>0$.
We have the following assumptions on the functions $w$ and $v$: 
\begin{itemize}
\item[{\bf (A1)}]  {\it
The weight function $w$ is nonnegative and continuous on $[0,h]$, and  satisfies}
\begin{equation}\label{eq:wc1}
\int_0^h w(x)\, dx=1, \qquad 
\mbox{and}\quad 
w(x) =0 \quad \forall x\notin [0,h].
\end{equation}
\end{itemize}

Note that $w$ can be discontinuous at $x=0$ and $x=h$. 
Although in \eqref{eq:claw1}-\eqref{eq:claw2} $w$ is only used on the interval $[0,h]$,
we define it on the whole real line for later use in the particle models. 
\begin{itemize}
\item[{\bf (A2)}] {\it
The velocity function $v\in \mathcal{C}^2$ satisfies}
\begin{equation}\label{eq:vc}
v(0)=1, \quad v(1)=0, \qquad \mbox{and}\quad { v'(\rho) < 0} \quad \forall \rho\in[0,1].
\end{equation}
\end{itemize}
Assumptions (A2) are commonly used in traffic flow, indicating that for higher density the cars travel with lower speed.

The conservation laws~\eqref{eq:claw1}-\eqref{eq:claw2}
are  {\em macroscopic} models for traffic flow.
They can be formally derived as the continuum limit of the corresponding 
{\em particle models}, commonly referred to also as  {\em microscopic} models. 
Particle models consist of systems of ODEs that describe 
the time evolution of the position of  each individual car.

\medskip

\textbf{Particle model for~\eqref{eq:claw1}.}  In connection with the conservation law~\eqref{eq:claw1}, we consider  the following particle model.
Assuming that all cars have the same length $\ell\in\mathbb{R}^+$, 
let  $z_i (t)$ be the position of the $i$-th car at time $t$.  
We order the indices for the cars so that 
\begin{equation}\label{zi}
z_i(t) \le z_{i+1}(t)-\ell\qquad \forall t\ge 0, \quad \mbox{for every}~i\in\mathbb{Z}.
\end{equation}
For the $i$-th car, we define the local traffic density perceived by its driver, 
depending on
the relative position of the car in front, namely
\begin{equation}\label{defrho}
\rho_i(t)  \; \dot= \; \frac{\ell}{z_{i+1}(t)-z_i(t)}
\qquad \forall t\ge 0, \quad \mbox{for every}~i\in\mathbb{Z}.
\end{equation}
Note that if $\rho_i= 1$, then the two cars with indices $i$ and $i+1$ are
bumper-to-bumper.   For the model to be meaningful,
we therefore must have $0 \le \rho_i(t)  \le 1$ for all $i\in\mathbb{Z}$ and $t\ge 0$. 

We consider the {\em ``follow-the-leaders'' (FtLs) model}, defined as follows.
The speed of the $i$-th car  depends solely on 
a weighted average local density $\rho^*_i$, where the average is taken over an interval 
of length $h$ in front of $z_i$.   
More precisely, denoting a time derivative with an upper dot, we assume
\begin{equation}\label{FtLs}
 \dot z_i(t)  = v(\rho^*_i(t)),\quad \mbox{with}\quad 
  \rho^*_i(t)\, \dot= \, \sum_{k=0}^{+\infty} w_{i,k}(t) \rho_{i+k}(t)\,.
\end{equation}
For $k\geq 0$,  by $w_{i,k}(t)$ we denote
the weight assigned at time $t$ by the $i$-th driver to the car at $z_{i+k}(t)$.
In connection with the weight function $w(\cdot)$ in \eqref{eq:claw1}--\eqref{eq:wc1},
these weights are defined as 
 \begin{equation}\label{eq:wik}
w_{i,k}(t)  \,\dot=\, \int_{z_{i+k}(t)}^{z_{i+k+1}(t)} w(y-z_i(t))\; dy,
\qquad k\geq 0,1,2,\cdots.
\end{equation}
Notice that the summation in \eqref{FtLs} actually contains only finitely many non-zero terms.
Indeed, if  $(m+1) \ell \ge h$, then for every $k>m$ one has
$$z_{i+k}(t)~\geq~z_i(t)+(m+1)\ell~\geq~z_i(t)+h.$$
Hence, by \eqref{eq:wc1}, $w_{i,k}=0$. 
With the above definition, from~\eqref{eq:wc1} it also follows
\begin{equation}\label{eq:wp}
\sum_{k=0}^m w_{i,k}(t)=1, \qquad  w_{i,k}(t)\ge 0 \qquad \forall t \ge 0. 
\end{equation}

\medskip

\textbf{Particle model for~\eqref{eq:claw2}.}  
In this model, the speed of the car number $i$ located at $z_i$ is
the weighted average of the function $v(\cdot)$ over an interval in front of $z_i$. 
This leads to the second FtLs model: 
\begin{equation}\label{FtLs2}
 \dot z_i(t)  = v_i^*(\rho;t),\qquad \mbox{where}\quad 
  v_i^*(\rho;t)\;\dot= \;\sum_{k=0}^{+\infty} w_{i,k} (t)\, v(\rho_{i+k}(t)).
\end{equation}
Here the weights $w_{i,k}(t)$ are defined as  in~\eqref{eq:wik}. 
As before, the sum contains only finitely many non-zero terms.

\bigskip

The convergence of microscopic models to their macroscopic equivalents 
is of fundamental interest. 
In this paper, we  study this issue by looking at traveling wave profiles, defined below.
Since the equations~\eqref{eq:claw1} and~\eqref{eq:claw2} are rather similar, 
as well as the systems~\eqref{FtLs} and~\eqref{FtLs2}, 
we analyze in detail~\eqref{eq:claw1} and~\eqref{FtLs}.
The analysis for~\eqref{eq:claw2} and~\eqref{FtLs2} will only be presented briefly.

For~\eqref{eq:claw1} and~\eqref{eq:claw2}, traveling wave solutions 
are special solutions of the Cauchy problem where a profile travels with a constant velocity. 

\begin{definition}\label{def:TW}
We say that  $Q:\mathbb{R}\mapsto [0,1]$ is a {\em ``traveling wave profile'' for~\eqref{eq:claw1} with speed $\sigma$} if the function $\rho(t,x) = Q(x-\sigma t)$ 
provides a solution to \eqref{eq:claw1}.
In the special case where $\sigma=0$, we call $Q(\cdot)$ a {\em ``stationary profile''}.
\end{definition} 

To simplify the discussion, throughout the sequel
we seek stationary profiles $Q(\cdot)$ for~\eqref{eq:claw1}.
Traveling waves with non-zero velocity can be transformed 
into a stationary profile using a coordinate translation (see Section 5.1). 
In Section~\ref{sec:NLCL}
we derive the {\em delay integro-differential equation} \eqref{eq:p2}
satisfied by a stationary profile $Q(\cdot)$. 
The existence and uniqueness (up to horizontal shifts)
of monotone profiles are established. 
The asymptotic stability  of the traveling wave profiles is
important in the study of the long-time behavior of solutions. 
We show that, under mild assumptions on the smooth initial condition,  
the profiles $Q(\cdot)$ are the time asymptotic limits of solutions of~\eqref{eq:claw1},
as $t\to +\infty$.

\medskip

In addition,
we also seek ``stationary  discrete wave profiles'' $P(\cdot)$ 
for the FtLs model~\eqref{FtLs}, defined as follows. 

\begin{definition}\label{def:DTW}  
We say that $P:\mathbb{R}\mapsto [0,1]$ is a {\em ``stationary discrete wave profile''} for~\eqref{FtLs} if there exists a solution  $\{z_i(t); i\in\mathbb{Z}\}$ of~\eqref{FtLs},
such that 
\begin{equation}\label{P1}
P(z_i(t)) = \rho_i(t) = \frac{\ell}{z_{i+1}(t)-z_i(t)}, 
\qquad \forall t\ge 0, \quad \forall i \in\mathbb{Z}.
\end{equation}
\end{definition}

We derive a {\em delay differential equation}
with a summation term, see~\eqref{eq:dP},  satisfied by $P(\cdot)$. 
In a similar way as for $Q(\cdot)$,  we establish 
the existence and uniqueness (up to a horizontal shift) of the discrete 
profiles $P(\cdot)$. 
Furthermore, we show that 
these profiles provide  attractors to the solutions of the FtLs model~\eqref{FtLs},
for a wide family of initial data.

The profile $P(\cdot)$ depends on the length of the cars  $\ell$. 
Taking the limit $\ell\to 0$,
we prove the convergence of traveling wave solutions $P(\cdot)$ for the particle model
to the profile $Q(\cdot)$ for the nonlocal conservation laws.

An entirely similar set of results is proved for the system~\eqref{FtLs2} 
and the conservation law~\eqref{eq:claw2}, with only small modifications in the analysis. 

\medskip

For the local follow-the-leader model, where the speed of each car
is determined solely by the leader ahead,  the traveling wave profiles have been  
studied in a recent work by Shen \& Shikh-Khalil~\cite{ShenKarim2017},
where existence, uniqueness and stability of traveling waves were established.
For the general Cauchy problem, convergence of the local FtL model 
to the corresponding local macroscopic PDE model 
\begin{equation}\label{eq:c1}
\rho_t + f(\rho)_x=0, \qquad   \mbox{where}\quad f(\rho) = \rho v(\rho),
\end{equation}
has been studied in various 
papers~\cite{MR3356989, HoldenRisebro, HoldenRisebro2}.

For the nonlocal model~\eqref{eq:claw1}, 
existence of entropy weak solutions for the Cauchy problem  was proved
in~\cite{BG2016} by utilizing the convergence of a finite difference scheme,
and in~\cite{MR3461737} by means of a finite volume scheme. 
Well-posedness of the solutions for the Cauchy problem is also established in~\cite{BG2016}.
A similar result was proved in~\cite{AmorimColomboTeixeira2015} 
for kernel functions in $\C^2(\R)\cap W^{2,\infty}(\R)$ instead of $\C^1([0,h])$. 
Similar nonlocal conservation laws with symmetric kernel functions have been studied 
by~\cite{Zumbrun1999},
and~\cite{BetancourtBuergerKarlsenTory2011} in the context of sedimentation modeling. 
Multi-dimensional versions and systems were studied as models for crowd dynamics
\cite{AggarwalColomboGoatin2015, ColomboLecureuxMercier2011, AggarwalGoatin2016, ColomboGaravelloLecureuxMercier2011,  ColomboGaravelloLecureuxMercier2012, CrippaLecureuxMercier2013}. 
Other conservation laws with nonlocal flux functions 
include models for slow erosion of granular flow~\cite{ShenZhang, AmadoriShen2012},
synchronization behavior~\cite{AmadoriHaPark2017},
and materials with fading memory~\cite{ChenChristoforou2007}.
An overview over conservation laws with several other types of nonlocal flux functions 
can be found in~\cite{DuKammLehoucqParks2012, ColomboMarcelliniRossi2016} 
and the references therein.
See~\cite{MR3670045} for a recent result on uniqueness and regularity results 
on nonlocal balance laws and \cite{MR381810} for multi-dimensional 
nonlocal balance laws with damping.
Further relevant references can be found in~\cite{MR2860409, MR3672998}.
For classical results on delay differential equations, 
we refer to~\cite{MR0477368, MR0141863}.

The paper is organized as follows. 
In Section~\ref{sec:NLF} we study the non-local FtLs model~\eqref{FtLs}, 
proving existence and uniqueness of the traveling wave profiles. 
We also 
show that these profiles are time asymptotic limits of more general solutions to
the FtLs model. 
Similar results are proved for the non-local PDE model~\eqref{eq:claw1} 
in Section~\ref{sec:NLCL}. 
In Section~\ref{sec:conv} we prove the convergence of the profiles of the FtLs model~\eqref{FtLs}  to
those of the PDE model~\eqref{eq:claw1}, as the car length $\ell$ tends to 0. 
In Section~\ref{sec:insta} we discuss the case of travelling waves with non-zero 
velocity, and we consider a couple of examples where the profiles are unstable.
In Section~\ref{sec:model2} we treat the alternative 
system~\eqref{FtLs2} and the conservation law~\eqref{eq:claw2}, 
and prove similar results. 
Finally, some concluding remarks are given in Section~\ref{sec:cr}.

\section{Non-local Follow-the-Leaders models}\label{sec:NLF}
\setcounter{equation}{0}

We consider the non-local FtLs model in~\eqref{FtLs}, i.e.
\bel{ODEi}
\dot z_i(t) = v\left(  \sum_{k=0}^m \int_{z_{i+k}(t)}^{z_{i+k+1}(t)} w(y-z_i(t)) \, dy \cdot
\frac{\ell}{z_{i+k+1(t)}-z_{i+k}(t)}\right), 
\qquad i\in\mathbb{Z}.
\eeq
Here $m$ is chosen so that $(m+1)\ell \geq h$.
This family of countably many ODEs can be 
regarded as a nonlinear dynamical system on an infinite dimensional space.
For example, we could set $y_i(t)=z_i(t)-z_i(0)$ and write the system \eqref{ODEi}
as an evolution equation on the Banach space of bounded sequences of real numbers
$y=(y_i)_{i\in\mathbb Z}$,
with norm $\|y\|= \sup_i |y_i|$.
For each $i\in\mathbb{Z}$, the right hand side of \eqref{ODEi} is Lipschitz continuous.
For a given initial datum, the existence and uniqueness of solutions to this system 
follow from standard theory of evolution equations in Banach spaces, 
see for example~\cite{MR0492671, MR1873467}. 

We now  derive the equation satisfied by a stationary profile $P(\cdot)$, 
considered in Definition~\ref{def:DTW}.
Note that~\eqref{FtLs} can be rewritten as a system of 
ODEs for the
discrete density functions $\rho_i(t)$, for $i\in\mathbb{Z}$,
\begin{align}\label{rhodot}
    \dot \rho_i(t) 
~=~-\frac{\ell \left( \dot z_{i+1} (t)- \dot z_i(t) \right) }{(z_{i+1}(t)-z_i(t))^2} 
~=~ \frac{1}{\ell} \rho_i^2(t) \cdot \Big[ v(\rho^*_{i}(t)) - v(\rho^*_{i+1}(t))\Big].
\end{align}

Differentiating both sides of~\eqref{P1} w.r.t.~$t$, 
and using~\eqref{FtLs} and~\eqref{rhodot},
one obtains
\begin{equation}\label{eq:dP}
P'(z_i) ~=~ \frac{\dot \rho_i}{\dot z_i} ~= ~
\frac{\rho_i^2}{\ell \cdot v(\rho^*_i)} \Big[v(\rho^*_i)-v(\rho^*_{i+1})\Big]
=
\frac{P^2(z_i)}{\ell \cdot v( P^*(z_i))} \Big[v( P^*(z_i))-v( P^*(z_{i+1}))\Big].
\end{equation}
Here and in the sequel, a prime denotes a derivative w.r.t.~the space variable $x$.
To shorten the notation, we do not explicitly write out the time dependence of
$z_i$ and $\rho_i$.
Furthermore,  $ P^*(z_i)$ is the weighted average of $P$ over an interval in front of $z_i$, 
defined as 
\begin{equation}\label{eq:P*}
 P^*(z_i) \;\dot=\; \sum_{k=0}^m w_{i,k} P(z_{i+k}).
\end{equation}

To proceed with the analysis, 
we need to introduce some  notations. 
Given a profile $P$ with $P(x) >0$ for every $x$, we define an operator $L^P(x)$ for the position of the leader
of the car at $x$,
\begin{equation}\label{eq:LP}
L^P(x) \;\dot=\; x + \frac{\ell}{P(x)}.
\end{equation}
 We also write
\[ 
(L^P)^2=L^P\circ L^P \quad \mbox{and}\quad 
(L^P)^k \;\dot=\; \underbrace{L^P \circ L^P \circ \cdots \circ L^P}_\text{k times}
\]
to denote the composition of $L^P$ with itself multiple times.
We then have
\[
z_{i+1}=L^P(z_i) ,\qquad
z_{i+2}=L^P(z_{i+1})=(L^P)^2 (z_i) =   z_i+ \frac{\ell}{P(z_i)} + \frac{\ell}{P(z_{i+1})} ,
\]
and for a general index $k\in\mathbb{N}$,
\[
z_{i+k} = (L^P)^k(z_i)  = z_i + \sum_{j=0}^{k-1}\frac{\ell}{P(z_{i+j})}.
\]

We also define an averaging operator $A^P$ as 
\begin{equation}\label{eq:AP}
A^P(z_i) \;\dot=\; \sum_{k=0}^m w_{i,k} P\left((L^P)^k(z_i) \right).
\end{equation}

Since $z_i$ is arbitrarily chosen, we now write $x=z_i$. We have
\begin{equation}\label{eq:P}
P'(x)~ =~ - \frac{P^2(x)}{\ell \cdot v(A^P(x)) } \Big[ v(A^P(L^P(x))) - v(A^P(x))\Big].
\end{equation}
We see that the profile $P$ satisfies a {\em delay differential equation}, 
where the delays are introduced 
in the right-hand side of~\eqref{eq:P} by the operators $L^P$ and $A^P$. 
Since $P(x)\in(0,1)$ for all $x$, according to \eqref{eq:LP}
 the delay in~\eqref{eq:P}  will always be larger than $\ell$.

We seek continuous and monotone profiles $P(\cdot)$ 
that satisfy~\eqref{eq:P} with given asymptotic conditions at $x\to\pm\infty$. 
In the analysis below  we also study the initial value problem, 
where the solutions might be non-differentiable at the initial point. 
Therefore,  the derivative $P'(x)$ on the lefthand side of~\eqref{eq:P} 
indicates $P'(x-)$.

\subsection{Technical lemmas}

We assume that $w(\cdot)$ satisfies~\eqref{eq:wc1} and 
$v(\cdot)$ satisfies~\eqref{eq:vc}.  
Let  $\hat \rho$ denote the unique stagnation point of the local conservation law~\eqref{eq:c1}, i.e., the  point where, for $f(\rho)\,\dot=\,\rho v(\rho)$, we have
\begin{equation}\label{eq:stag}
0<\hat\rho<1, \quad 
f'(\hat \rho)=0, \quad \mbox{and}\quad 
\begin{cases} 
f'(\rho)>0 &~\mbox{for} ~\rho<\hat\rho, \\
 f'(\rho)<0 &~\mbox{for} ~\rho>\hat\rho. 
 \end{cases}
\end{equation}

The existence of solutions of~\eqref{eq:P} will be established in Section~\ref{sec:2.2}
for initial value problem and Section~\ref{sec:2.3} for asymptotic value problems. 
Assuming that solutions exist,  in this section we establish several technical lemmas.
We start with a definition. 

\begin{definition}\label{def:1}
Let the function $P:\mathbb{R}\mapsto (0,1) $ be given,
and let $\ell \in\mathbb{R}^+$ be the length of each car. 
Assume that 
\begin{equation}\label{PPc}
\ell P'(x)< P^2(x)\qquad  \forall x\in\mathbb{R}.
\end{equation}
We call a sequence of car positions  $\{z_i\,;~ i\in\mathbb{Z}\}$
{\em ``a distribution generated by $P(\cdot)$''} if
\begin{equation}\label{eq:def1}
z_{i+1}-z_i = \frac{\ell}{P(z_i)}, \qquad \forall i\in\mathbb{Z}.
\end{equation}
\end{definition}

The assumption~\eqref{PPc} ensures that for each car position $z_i$, 
there exists a unique follower $z_{i-1}$ such that 
$z_{i-1} + \ell/P(z_{i-1}) =z_i$. Furthermore, it implies that
\[ x+\ell/P(x) > y+\ell/P(y) \qquad \mbox{for every}  ~x>y.\]
Note that if $P(\cdot)$ satisfies the equation~\eqref{eq:P}, 
then~\eqref{PPc} holds, because 
\begin{equation}\label{PPcc}
P'(x) = \frac{P^2(x)}{\ell}  \cdot \left[ 1 - \frac{v(A^P(L^P(x)))}{v(A^P(x))}\right]
\le \frac{P^2(x)}{\ell}.
\end{equation}

We further note that
there exist infinitely many car distributions for any given profile $P$. 
However, if we fix the position of one car, say $z_0$, then the distribution is unique.

\begin{lemma}[Asymptotic limits]\label{lm:A}
Assume that $P(\cdot)$ is a bounded solution of~\eqref{eq:P} whose 
asymptotic limits satisfy 
\[
\lim_{x\to-\infty} P(x) = \rho^-,\qquad 
\lim_{x\to +\infty} P(x) = \rho^+,\qquad \lim_{x\to\pm\infty} P'(x)=0,
\]
where $\rho^-,\rho^+\in\mathbb{R}$ and $\rho^-, \rho^+\in(0,1)$ are real numbers.
Then, the following holds.
\begin{itemize}
\item
As $x\to+\infty$, $P(x)$  approaches $\rho^+$ with an exponential rate $\lambda_+^\ell \in\mathbb{R}^+$ 
if and only if  $\rho^+ > \hat\rho$,
where $\hat\rho$ satisfies~\eqref{eq:stag}.
The rate $\lambda_+^\ell$  satisfies the estimate
\begin{equation}\label{lambda+e}
\lambda^\ell_+ > \frac{b-1}{bh+a},
\qquad \mbox{where}\quad 
a\;\dot=\; \frac{\ell}{\rho^+},\quad b\;\dot=\; - \frac{\rho^+ v'(\rho^+)}{ v(\rho^+)}.
\end{equation}
\item
As $x\to-\infty$, $P(x)$  approaches $\rho^-$ with an exponential rate $\lambda_-^\ell\in\mathbb{R}^+$ 
if and only if  $\rho^- < \hat\rho$, where $\hat\rho$ satisfies~\eqref{eq:stag}.
The rate $\lambda_-^\ell$  satisfies the estimate
\begin{equation}\label{lambda-e}
\lambda^\ell_- > \frac{b'-1}{b'h+a'},
\qquad \mbox{where}\quad 
a'\;\dot=\; \frac{\ell}{\rho^-},\quad b'\;\dot=\; - \frac{\rho^- v'(\rho^-)}{ v(\rho^-)}.
\end{equation}
\end{itemize}
\end{lemma}

\begin{proof} 
We consider the limit $x\to+\infty$, and linearize~\eqref{eq:P} at $\rho^+$. 
We write
\[ P(x) = \rho^+ + \eta(x),\]
where $\eta(\cdot)$ is a first order perturbation.  Below we use ``$\approx$'' to denote
the first order approximation. 
Let $\{z_i: i\in\mathbb{Z}\}$ be a car distribution generated by $P(\cdot)$,
we have,
\[ L^P(x) \approx   x + \frac{\ell}{\rho^+} , \qquad 
z_{i+k+1} - z_{i+k} \approx \frac{\ell}{\rho^+}, \qquad
(L^P)^k(x) \approx  x + \frac{k \ell}{\rho^+}.
\]
For the weights $w_{i,k}$ we have 
\begin{equation}\label{eq:wkhat}
w_{i,k} \approx \int_{z_{i+k}}^{z_{i+k} + \ell/\rho^+} w(y-z_i)\; dy
\approx \int_{k\ell/\rho^+}^{(k+1)\ell/\rho^+} w(s)\; ds \;\dot=\; \hat w_k.
\end{equation}
Note that the approximated weights $\hat w_k$ are independent of the index $i$. 
We have 
\[
\hat w_k \ge 0 \quad (0\le k \le m),  \qquad
 \sum_{k=0}^{m} \hat w_k =1.
\]
We compute
\begin{eqnarray*} 
A^P(x) &\approx& 
\sum_{k=0}^{ m} \hat w_{k} \left[\rho^+ + \eta\left(x+\frac{k\ell}{\rho^+}\right)\right]
= \rho^+ + \sum_{k=0}^{ m} \hat w_{k} \cdot\eta\left(x+\frac{k\ell}{\rho^+}\right),\\
A^P\left(L^P(x)\right) &\approx&
 \rho^+ + \sum_{k=0}^{ m} \hat w_{k} \cdot\eta\left(x+\frac{(k+1)\ell}{\rho^+}\right),\\
v\left(A^P(x)\right) &\approx& 
v(\rho^+) + v'(\rho^+) 
\cdot \sum_{k=0}^{ m} \hat w_{k} \cdot\eta\left(x+\frac{k\ell}{\rho^+}\right),
\\
v\left(A^P(L^P(x))\right)& \approx &
v(\rho^+) + v'(\rho^+) \cdot 
\sum_{k=0}^{ m} \hat w_{k} \cdot\eta\left(x+\frac{(k+1)\ell}{\rho^+}\right).
\end{eqnarray*}
Plugging all these approximations into~\eqref{eq:P}, 
we obtain
\begin{eqnarray*}
\eta'(x) &\approx& 
-v'\left(\rho^+\right) \frac{(\rho^+)^2}{\ell \cdot v(\rho^+)}  \sum_{k=0}^{ m} \hat w_{k} 
\cdot \left[
\eta\left(x+\frac{(k+1)\ell}{\rho^+}\right) - \eta\left(x+\frac{k\ell}{\rho^+}\right)
\right].
\end{eqnarray*}
Using the positive coefficients $a,b$ defined in~\eqref{lambda+e}, we can write
the linearization of~\eqref{eq:P} as 
\begin{equation}\label{eq:eta}
\eta'(x) = \frac{b}{a} \cdot \sum_{k=0}^{ m} \hat w_{k} \Big[
\eta(x+(k+1)a) - \eta(x+ka)
\Big],
\end{equation}
where the linearized weights $\hat w_{k} $ are given in~\eqref{eq:wkhat}. 

Note that \eqref{eq:eta} is a {\em linear delay differential equation}, which can be solved explicitly using its characteristic equation.
Seeking solutions of the form 
\[
\eta(x) = M e^{-\lambda x}, \qquad \lambda\in\mathbb{R},
\]
where $M$ is an arbitrary constant (positive or negative), 
we have the characteristic equation
\[
-\lambda = \frac{b}{a} \cdot \sum_{k=0}^{ m} \hat w_{k}  
\left[  e^{-(k+1)a \lambda} - e^{-ka\lambda}\right]
= - \frac{b}{a}  \left(1- e^{-a\lambda}\right)\cdot \sum_{k=0}^{ m} \hat w_{k}   e^{-ka\lambda} .
\]
Thus, 
the rate $\lambda$ must satisfy the equation 
\begin{equation}\label{eq:G22}
\mathcal{L}(\lambda) = \mathcal{R}(\lambda), \quad \mbox{where}\quad
\begin{cases} \displaystyle
\mathcal{L}(\lambda) \;\dot=\;  b \sum_{k=0}^{ m} \hat w_{k}   e^{-ka\lambda},  \\[4mm]
\displaystyle
  \mathcal{R}(\lambda) \;\dot=\;  \frac{a\lambda}{ 1-e^{-a\lambda}} , 
\quad \mbox{and}\quad
\mathcal{R} (0) \;\dot=\; \lim_{\lambda \to 0} \mathcal{R} (\lambda) = 1. 
\end{cases}
\end{equation}
The functions $\mathcal{L}(\cdot)$ and $\mathcal{R}(\cdot)$ are continuous, 
satisfying the properties
\begin{eqnarray*}
\mathcal{L}(0)= b, \qquad \lim_{\lambda\to\infty}\mathcal{L}(\lambda)=0,
\qquad
\lim_{\lambda\to-\infty}\mathcal{L}(\lambda)=\infty, 
\quad &\mbox{and}&
\mathcal{L}'(\lambda) <0
\quad\forall \lambda\in\mathbb{R}, \\
\mathcal{R}(0)= 1, \qquad 
\lim_{\lambda\to\infty}\mathcal{R}(\lambda)=\infty,
\qquad \lim_{\lambda\to- \infty}\mathcal{R}(\lambda)=0,
\quad
&\mbox{and}&
 \mathcal{R}'(\lambda) >0\quad\forall \lambda\in\mathbb{R}.
\end{eqnarray*}
We see that $\mathcal{L}(\cdot)$ is monotonically decreasing and $\mathcal{R}(\cdot)$ is monotonically increasing, 
and the range of both functions is $(0,\infty)$. 
We conclude that there exists exactly one solution $\lambda$ for~\eqref{eq:G22}. 
Furthermore, we observe that:
\begin{itemize}
\item
If $b>1$, the solution $\lambda$ is positive, which we denote by  $\lambda_+^\ell >0$; 
\item
if $b=1$, the solution is $\lambda=0$, which leads to the trivial solution $P(x)\equiv \rho^+$;
\item
if $b<1$, the solution $\lambda$ is negative, which leads to an unstable asymptote. 
\end{itemize}
Thus, we obtain a stable asymptote at $x\to +\infty$ if and only if $b>1$.
We have 
\[
b>1 \quad \Longleftrightarrow \quad 
-v'(\rho^+) \frac{(\rho^+)}{ v(\rho^+)}>1 \quad \Longleftrightarrow \quad 
f'(\rho^+)=v(\rho^+)+\rho^+v'(\rho^+) <0  \quad \Longleftrightarrow \quad 
\rho^+ > \hat \rho,
\]
where $f(\rho)=\rho v(\rho)$, and $\hat\rho$ satisfies~\eqref{eq:stag}.

\medskip

To get an estimate on $\lambda^\ell_+$, we define the monotone function
$H(\lambda)\,\dot=\,\mathcal{L}(\lambda)-\mathcal{R}(\lambda)$. 
Since $b>1$, we have 
\[
H(0) = b-1 >0, \qquad  H'(\lambda) <0  \quad \forall \lambda\in\mathbb{R},
\]
therefore $H$ has a unique zero which is positive. 
Using the inequality 
\[
\frac{x}{1-e^{-x}} < 1+x \qquad \forall x>0,
\] 
we obtain 
\begin{equation}\label{estR}
-\mathcal{R}(\lambda)  > -(1+a\lambda) \qquad \mbox{for}~~\lambda >0.
\end{equation}
Moreover, since $\hat w_k=0$ if $ ka  \ge h $, 
then for $\lambda>0$ we have 
\begin{equation}\label{estL}
\mathcal{L}(\lambda) > b \sum_{k=0}^m \hat w_k e^{-h\lambda} = b e^{-h\lambda}.
\end{equation}
Combining~\eqref{estR}-\eqref{estL}, we have
\begin{equation}\label{cH}
H(\lambda) > \mathcal{H}(\lambda) \;\dot=\;b e^{-h \lambda} -1 - a \lambda  
\qquad \forall \lambda>0. 
\end{equation}
The function $\mathcal{H}(\cdot)$ has the properties
\[
\mathcal{H}(0) = b-1,  \quad
 \mathcal{H}'(\lambda) = -b he^{-h\lambda} -a  <0 , \quad
 \mathcal{H}''(\lambda) = bh^2 e^{-h\lambda} >0 \qquad \forall \lambda>0.
\]
Therefore $\mathcal{H} $ has a unique positive root $\lambda^\flat$, 
satisfying the rough estimate
\[
\lambda^\flat > -\frac{\mathcal{H}(0)}{\mathcal{H}'(0)} = \frac{b-1}{bh+a}.
\]
Using the fact that $\lambda^\ell_+ > \lambda^\flat$, we obtain the estimate~\eqref{lambda+e}.

A completely similar computation can be carried out for the limit $x\to-\infty$, 
replacing $\rho^+$ with $\rho^-$.  We omit the details.
\end{proof}

\begin{remark}\label{rm:1}
We see that, if $\rho^+=\hat \rho$, then $b=1$ and so $\lambda^\ell_+=0$.
If $\rho^+\to 1$, then $b\to \infty$, so~\eqref{eq:G22} implies 
\[
 \sum_{k=0}^{ m} \hat w_{k}   e^{-ka\lambda_+} \to 0 \qquad 
 \Rightarrow \quad \lambda_+ \to\infty.
\] 
A similar argument shows that if $\rho^-\to0$, then $\lambda_- \to \infty$. 
Therefore, we obtain two trivial cases for the profile $P(\cdot)$:
\begin{itemize}
\item If  $\rho^-=\hat\rho=\rho^+$, then $P(x)\equiv \hat\rho$ for all $x\in\mathbb{R}$;
\item If $\rho^-=0$ and $\rho^+=1$, then $P$ is a unit step function, with a jump at some $x_0\in\mathbb{R}$.
\end{itemize}
\end{remark}
Thanks to Remark~\ref{rm:1}, in the sequel 
we consider only the nontrivial cases:
\begin{equation}\label{eq:Assu}
0 < P(x) < 1 \quad \forall x\in\mathbb{R}, \qquad \mbox{and}\quad  0<\rho^- < \hat \rho< \rho^+<1.
\end{equation}

\begin{definition}\label{def:period}
Let $\{z_i(t): i\in\mathbb{Z}\}$ be the solution of~\eqref{FtLs} with initial condition $\{z_i(0): i\in\mathbb{Z}\}$.
We say that $\{z_i(t): i\in\mathbb{Z}\}$ is \textbf{periodic} if there exists a constant $t_p\in\mathbb{R}^+$,
independent of $i$ and $t$, such that
\begin{equation}\label{PT}
z_i (t+t_p) = z_{i+1}(t), \qquad \forall i\in\mathbb{Z}, ~ \forall t \ge 0.
\end{equation}
We refer to $t_p$ as the \textbf{period}.
\end{definition}

Definition~\ref{def:period} indicates that, in a periodic solution $\{z_i(t): i\in\mathbb{Z}\}$, 
after a time period of $t_p$, each car takes over the position of its leader. 
Intuitively,  if $P$ is a stationary profile, and $\{z_i(t)\}$ satisfies~\eqref{P1}, 
then $\{z_i(t)\}$ must be periodic. 
Indeed, in next Lemma we show that these two situations are equivalent.

\begin{lemma} \label{lm:period}
Let $P(\cdot)$ be a continuous function.
Then, $P(\cdot)$ satisfies~\eqref{eq:P} if and only if
\begin{equation}\label{eq:PE}
\int_x^{x+\ell/P(x)} \frac{1}{v(A^P(z))}\; dz = t_p,\qquad \forall x\in\mathbb{R}
\end{equation}
for some constant $t_p \in \mathbb{R}^+$. 
Moreover we have
\begin{eqnarray}
\lim_{x\to \infty}P(x) = \rho^+ \qquad 
&\implies& \quad  t_p=\frac{\ell}{\bar f}, ~ \mbox{where} ~ \bar f = f(\rho^+) ,
\label{eq:tp1}\\
\lim_{x\to -\infty}P(x) = \rho^- \qquad 
&\implies& \quad  t_p=\frac{\ell}{\bar f}, ~ \mbox{where} ~ \bar f = f(\rho^-) .
\label{eq:tp2}
\end{eqnarray}
Furthermore, let $\{z_i(t): i\in\mathbb{Z}\}$ be a solution of~\eqref{FtLs}. Then, 
$\{z_i(t)\}$ satisfies~\eqref{P1} if and only if $\{z_i(t)\}$ is periodic, with period $t_p$
given in~\eqref{eq:PE}.
\end{lemma}

\begin{proof}
\textbf{1.} Differentiating~\eqref{eq:PE} in $x$ on both sides, we obtain 
\[
\left(1- \frac{\ell}{P^2(x)} P'(x)\right)\frac{1}{v(A^P(L^P(x)))}  - \frac{1}{v(A^P(x))}=0, 
\]
which is equivalent to~\eqref{eq:P}.

\textbf{2.}
Furthermore, since~\eqref{eq:PE} holds for all $x\in\mathbb{R}$, we take the limit $x\to\infty$ and get
\[
t_p = \lim_{x\to\infty}\int_x^{x+\ell/P(x)} \frac{1}{v(A^P(z))}\; dz =
\frac{\ell}{\rho^+} \cdot \frac{1}{v(\rho^+)} = \frac{\ell}{f(\rho^+)},
\] 
proving~\eqref{eq:tp1}. 
A completely similar argument leads to~\eqref{eq:tp2}.

\textbf{3.} Let $\{z_i(t): i\in\mathbb{Z}\}$ be a solution of~\eqref{FtLs}. 
Assume that $\{z_i(t)\}$ satisfies~\eqref{P1}. 
Fix an index $i\in\mathbb{Z}$ and a time $\hat t\ge 0$. 
We have 
\[\frac{d z_i}{dt}   = v(A^P(z_i)) \qquad \implies \quad  \frac{d z_i}{v(A^P(z_i))} = dt,\]
a separable equation which can be solved implicitly.
The time it takes for the car at $z_i$ to reach $z_{i+1}=z_i+\ell/P(z_i)$ is
\[
t_{p,i} =\int_{\hat t}^{\hat t +t_{p,i}} dt =  \int_{z_i}^{z_i +\ell/P(z_i)} \frac{1}{v(A^P(z))}\; dz .
\]
Thanks to~\eqref{eq:PE}, we conclude $t_{p,i} \equiv t_p$ for all $i\in\mathbb{Z}$,
therefore $\{z_i(t)\}$ is periodic. 

On the other hand, assume that $\{z_i(t)\}$ is periodic, and let $t_p$ denote its period.
Fix a time $\hat t\in\mathbb{R}^+$, and consider the interval $t\in[\hat t , \hat t+t_p]$.
On each space interval  $x\in[z_i(\hat t), z_{i+1}(\hat t)]$, we define the function $P$ as  
\[
P(z_i(t)) = \rho_i(t) \qquad  t\in[\hat t , \hat t+t_p].
\]
Since $\{z_i(t)\}$ is periodic, we have 
$P(z_i(\hat t +t_p))=P(z_{i+1}(\hat t))$, 
and $P$ is continuous. 
The period, i.e., the time it takes for car at $z_i(\hat t)$ to reach $z_{i+1}(\hat t)$,
is
\[
t_p = \int_{z_i}^{z_i +\ell/P(z_i)} \frac{1}{v(A^P(z))}\; dz \quad \forall i .
\]
Since $\hat t$ is arbitrarily chosen, we conclude~\eqref{P1}.
\end{proof}

\subsection{Initial value problems} \label{sec:2.2}

We assume that the assumptions~\eqref{eq:wc1}-\eqref{eq:vc} hold. 
Fix any point $x_0\in\mathbb{R}$, and let $\Psi:[x_0, +\infty[\,\mapsto\R$ be 
a smooth increasing function,  satisfying
\begin{equation}\label{eq:P0}
0 <\Psi(x) < 1, \quad \Psi'(x) >0, \quad \forall x\ge x_0.
\end{equation}

We let the ``initial condition''  be given on $x\ge x_0$, i.e.,
\begin{equation}\label{P:int}
P(x)=\Psi(x),\qquad \mbox{for}~x\ge x_0.
\end{equation}
Then,  the delay differential equation~\eqref{eq:P} can be solved backwards in $x$, 
for $x<x_0$. 
We refer to this as the ``initial value problem'' for~\eqref{eq:P}, 
and we seek continuous solutions $P(\cdot)$ on $x<x_0$. 
Note that, since $\Psi(x)$ might not satisfy~\eqref{eq:P} on $x\ge x_0$,  
the derivative  $P'$ might not be continuous at~$x_0$. 
Therefore, it is assumed that the derivative $P'(x)$ in~\eqref{eq:P} denotes the left
derivative $P'(x-)$. 

\begin{lemma}[Monotonicity and positivity]\label{lm:mono}
Let $P(\cdot)$ be a solution of the initial value problem for~\eqref{eq:P},  
with initial condition~\eqref{P:int},  satisfying~\eqref{eq:P0}.  Then, we have 
\begin{equation}\label{eq:Pp}
0 <P(x) < \Psi(x_0), \quad P'(x) >0, \quad \forall x < x_0.
\end{equation}
\end{lemma}

\begin{proof} 
We first prove the monotonicity.  Observe that, since $\Psi$ is monotone increasing 
on $x\ge x_0$, by~\eqref{eq:P} we have $P'(x_0-)>0$. 
We now proceed with contradiction. 
Assume that $P(\cdot)$ fails to be monotone on $x<x_0$. Then there exists
a point $\hat x < x_0$ such that 
\begin{equation}\label{eq:Pcontra}
P'(\hat x)=0, \qquad P'(x) >0 \quad \forall x > \hat x. 
\end{equation}
Since $P(x)$ is monotone on $x>\hat x$, and $A^P$ is an averaging operator, 
we have 
\[
A^P(\hat x) < A^P\left(L^P(\hat x)\right),\qquad 
v\left(A^P(\hat x)\right) > v\left(A^P(L^P(\hat x))\right).
\]
By~\eqref{eq:P} we get $P'(\hat x) >0$, contradicting~\eqref{eq:Pcontra}. 

The positivity of $P(\cdot)$ follows from the fact 
that equation~\eqref{eq:P} is ``autonomous'' and $P=0$ 
is a critical point. 
\end{proof}

The next theorem states the existence and uniqueness for the  initial value problem. 

\begin{theorem}\label{th:2f}
Consider the initial value problem for~\eqref{eq:P} with initial  condition~\eqref{P:int}, 
satisfying~\eqref{eq:P0}.
Then, there exists a unique continuous solution $P(\cdot)$ on $(-\infty,x_0]$
with the following properties.
\begin{itemize}
\item  $P(\cdot)$ is monotone and Lipschitz continuous, with 
\begin{equation}\label{P-Lip}
P'(x)  \le \ell^{-1} P^2(x) 
\le \ell^{-1} \Psi^2(x_0) <\ell^{-1}.
\end{equation}

\item 
$P(\cdot)$ has the  asymptotic value  
\[
\lim_{x\to-\infty} P(x)=\rho^-_0, 
\]
where $\rho^-_0\in\mathbb{R}$ satisfies
\begin{equation}\label{eq:Ptp}
0<\rho^-_0< \hat\rho, \qquad 
\frac{\ell}{f(\rho^-_0)} = \int_{x_0}^{x_0+\ell/\Psi(x_0)} \frac{1}{v(A^\Psi(z))}\;dz.
\end{equation}
Here $\hat\rho$ is  defined in~\eqref{eq:stag}, and
the operator $A^\Psi$ is defined in~\eqref{eq:AP}, replacing $P$ with $\Psi$.

\item  
Let $L_v$ be the Lipschitz constant for the map $\rho \mapsto (1/v(\rho))$, 
we have 
\begin{equation}\label{eq:LDP}
P\left(x+\frac{\ell}{P(x)} + h\right) - P(x) > \frac{P(x)}{L_v} \cdot
\left[ \frac{1}{ f(\rho_0^-)} - \frac{1}{f(P(x))}\right].
\end{equation}
\end{itemize}
\end{theorem}

\begin{proof} 
\textbf{(1).} 
Since~\eqref{eq:P} is a delay differential equation  
with delay of at least $\ell$, the
existence and uniqueness can be  established by the standard method of steps, 
see~\cite{MR0477368}. 
We define the intervals 
\[
I_k \;\dot=\; [x_0-(k+1)\ell, ~x_0-k \ell], \qquad k\in\mathbb{Z}^+.
\]
Consider first the interval $I_0=[x_0-\ell,x_0]$. 
For $x\in I_0$, the right hand side of~\eqref{eq:P} involve only values with $x\ge x_0$,
which are given as initial condition.  The existence and uniqueness of solutions follow from 
standard theory for scalar ODEs. 
Furthermore, since $\Psi$ is monotone and positive, by Lemma~\ref{lm:mono} 
the solution $P(\cdot)$ is monotone and positive on $I_0$.
Finally, thanks to~\eqref{PPcc} we have
\[
P'(x)  \le \ell^{-1} P^2(x)  \le \ell^{-1} \Psi^2(x_0),
\]
proving~\eqref{P-Lip}.

\textbf{(2).}  
The argument can be repeated on  all subsequent intervals $I_k$ for $k\in\mathbb{Z}^+$,
leading to the existence and uniqueness of Lipschitz solution on  $x\le x_0$. 

\textbf{(3).}  
Furthermore, by the periodic property in Lemma~\ref{lm:period}, we have 
\[
\int_{x}^{x+\ell/P(x)} \frac{1}{v(A^P(z))}\;dz = t_p, \qquad \forall x \le x_0.
\]
Setting $x=x_0$, we have 
\[
t_p = \int_{x_0}^{x_0+\ell/\Psi(x_0)} \frac{1}{v(A^\Psi(z))}\;dz,
\]
and together with Lemma~\ref{lm:period}, we obtain~\eqref{eq:Ptp}. 

\textbf{(4).} By the periodic property in step (3), we have, for every $x<x_0$, 
\[
\frac{\ell}{\bar f_0} = \int_x^{x+\ell/P(x)}  \frac{1}{v(A^{P}(z))} dz , \qquad \bar f_0=f(\rho^-_0).
\]
Subtracting from it the identity
\[
\frac{\ell}{f(P(x))} = \frac{\ell}{P(x) v(P(x))}  =
\int_x^{x+\ell/P(x)}  \frac{1}{v(P(x))} dz,
\]
we get
\begin{equation}\label{vv}
\frac{\ell}{\bar f_0} - \frac{\ell}{f(P(x))} =
\int_x^{x+\ell/P(x)} \left[  \frac{1}{v\left(A^{P}(z)\right)} -\frac{1}{v(P(x))} \right] dz .
\end{equation}
Since $P$ is monotone increasing, we have, for any $z\in[x, x+\ell/P(x)]$, 
\[
A^{P}(z) <  A^{P}(x+\ell/P(x)) < P(x+\ell/P(x) + h).
\]
Combining this with~\eqref{vv},  we get
\begin{eqnarray*}
\frac{\ell}{\bar f_0} - \frac{\ell}{f(P(x))}
&< &
\frac{\ell}{P(x)}
\left[\frac{1}{v(P(x+\ell/P(x) + h))} - \frac{1}{v(P(x))} \right] \\
&< &
 \frac{\ell}{P(x)}\cdot L_v \cdot
\left[P(x+\ell/P(x) + h) - P(x)\right] ,
\end{eqnarray*}
where $L_v$ is the Lipschitz constant for the map $(1/v)$.
This proves~\eqref{eq:LDP}.
\end{proof}

\subsection{Asymptotic value problems}\label{sec:2.3}

\begin{theorem}[Asymptotic value problem]\label{th2b}
Let $w(\cdot)$ satisfy~\eqref{eq:wc1} and let $v(\cdot)$ satisfy~\eqref{eq:vc}.
Consider the asymptotic value problem for~\eqref{eq:P},
with asymptotic conditions
\begin{equation}\label{eq:BC}
\lim_{x\to -\infty} P(x)~=~\rho^-,\qquad 
\lim_{x\to \infty} P(x) ~=~\rho^+.
\end{equation}
If $\rho^-,\rho^+ \in\mathbb{R}$ satisfy
\bel{rpm}
0<\rho^-< \hat\rho<\rho^+<1, \qquad f(\rho^-)=f(\rho^+)=\bar f,
\eeq
then there exists  a monotone increasing, Lipschitz continuous solution $P(\cdot)$,
defined for all  $x\in\R$.

Furthermore, the solutions are unique up to horizontal shifts, in the following sense. 
If $P_1$ and $P_2$ are two solutions of the same asymptotic value problem,
then there exists a constant $c\in\mathbb{R}$ such that 
$ P_1 (x) = P_2(x+c)$ for all $x\in\mathbb{R}$.
\end{theorem}

\begin{proof}
\textbf{Existence of solutions.}  The proof for the existence of solutions take several steps.

{\bf 1.}  A solution will be constructed by taking the limit of a sequence of approximations.
Let $\lambda_+^\ell$ be the exponential rate given in Lemma~\ref{lm:A}
for the asymptotic condition $\lim_ {x\to+\infty}P(x)=\rho^+$,  and let
\begin{equation}\label{eq:Psi}
\Psi(x) = \rho^+ - e^{-\lambda_+^\ell x}.
\end{equation}
Let $\{x_n: n\in\mathbb{N}, x_n\in\mathbb{R}\}$  
satisfy $x_n<x_{n+1}$ for all $n$, and $\lim_{n\to\infty} x_n = \infty$, 
and let $P^{(n)}(\cdot)$ be the unique solution for the initial value problem of~\eqref{eq:P} 
with initial condition $P^{(n)}(x)=\Psi(x)$  for $x\ge x_n$,
established in Theorem~\ref{th:2f}. 
Then $P^{(n)}(\cdot)$ is Lipschitz,  
positive and monotone increasing for  $x\in \,]-\infty, x_n]$. 
Denoting
\[
\rho^-_n \;\dot=\; \lim_{x\to-\infty} P^{(n)}(x),
\]
by Lemma~\ref{lm:A} we have $\rho^-_n < \hat\rho$. 
We further claim that 
\begin{equation}\label{eq:ff}
\lim_{n\to \infty}  f(\rho^-_n) =  f(\rho^+) = f(\rho^-).
\end{equation}
Indeed, denoting 
\[
\bar f_n \;\dot=\; \ell \cdot 
\left[\int_{x_n}^{x_n+\ell/\Psi(x_n)} \frac{1}{v(A^\Psi(z))}\;dz\right]^{-1}, 
\]
by Theorem~\ref{th:2f}, we have 
\[
\frac{\ell}{f(\rho^-_n)} = \frac{\ell}{\bar f_n}.
\]
Let $\ve>0$. There exists an $N$, sufficiently large,  such that 
$e^{-\lambda_+^\ell x_n} < \ve$ for all $n>N$. 
Using~\eqref{eq:Psi}, we compute, for $n>N$, 
\begin{eqnarray*}
\frac{\ell}{f(\rho^-_n)} ~<~ \int_{x_n}^{x_n+\ell/(\rho^+-\ve)} \frac{1}{v(\rho^+)}\;dz
&=&\frac{\ell}{(\rho^+-\ve)v(\rho^+)}, \\
\frac{\ell}{f(\rho^-_n)} ~>~\int_{x_n}^{x_n+\ell/\rho^+} \frac{1}{v(\rho^+-\ve)}\;dz
&=& \frac{\ell}{\rho^+v(\rho^+-\ve)}. 
\end{eqnarray*}
Since $\ve>0$ is arbitrary, we conclude that 
\[ \lim_{n\to\infty} \frac{\ell}{\bar f_n} =
\lim_{n\to\infty} \frac{\ell}{f(\rho^-_n)} = \frac{\ell}{\rho^+ v(\rho^+)}=\frac{\ell}{f(\rho^+)},
\]
which proves~\eqref{eq:ff}. This further implies that 
\bel{r-l}
\lim_{n\to\infty} \rho^-_n=\rho^-, \qquad
\lim_{n\to\infty} \bar f_n = \bar f = f(\rho^-) = f(\rho^+).
\eeq
\v
{\bf 2.} By Theorem~\ref{th:2f}, $P^{(n)}$ is Lipschitz continuous, with Lipschitz
constant $\ell^{-1}$. 
Since we are assuming $\rho^-<\hat\rho<\rho^+$, by \eqref{r-l} and \eqref{P-Lip}
it follows that, for every $n$ large enough, there exists some point $\xi_n$ 
such that $P^{(n)}(\xi_n)= \hat\rho$.
We can thus consider the sequence of shifted profiles 
\bel{HPn} \Hat P^{(n)}(x)~\doteq~P^{(n)}(x-\xi_n).\eeq
This guarantees that $\Hat P^{(n)}(0)=\hat\rho$, for every $n$ large enough.
\v
{\bf 3.} Since all functions $\Hat P^{(n)}$ are uniformly bounded and increasing, 
using Helly's compactess theorem,  by possibly taking a subsequence  we obtain
the pointwise convergence 
\bel{pnp}\Hat P^{(n)}(x)~\to~P(x)\eeq
for some limit function $P(\cdot)$.  Since the functions $ \Hat P^{(n)}(\cdot)$ are 
also uniformly Lipschitz continuous, by the 
Arzel\`{a}-Ascoli theorem   the convergence \eqref{pnp}  is uniform for $x$ in bounded
intervals.  

From the properties of all  $ \Hat P^{(n)}$ it immediately follows that $P$ 
is nondecreasing and uniformly Lipschitz continuous.
Moreover 
\bel{pn0}P(0) ~=~ \lim_{n\to\infty} \Hat P^{(n)}(0)~=~\hat\rho.\eeq 

By Lemma~\ref{lm:period}, the differential equation \eqref{eq:P}  
can be written in the integral form for $\Hat P^{(n)}$,
\[{\color{black}
\int_x^{x+\ell/\Hat P^{(n)}(x)}  v^{-1}(A^{\Hat P^{(n)}} (z))\; dz = \bar f_n}.
\]
Recalling that the convergence $\Hat P^{(n)}(x)~\to~P(x) $ is uniform on 
 bounded sets, we conclude that 
the limit function $P$ satisfies the integral equation
\[
\int_x^{x+\ell/ P(x)}  v^{-1}(A^{ P} (z))\; dz = \bar f.
\]
By Lemma~\ref{lm:period},  $P$  provides a solution to~\eqref{eq:P}.

\medskip

{\bf 4.} It remains to prove the asymptotic limits 
\bel{Plim}\lim_{x\to -\infty}P(x)~=~\rho^-,\qquad\qquad\lim_{x\to +\infty}P(x)~=~\rho^+.\eeq
Since $P(\cdot)$ is nondecreasing and bounded, it is clear that these two limits exists.
Assume that 
\[ 
P^+ \;\dot=\; \lim_{x\to\infty} P(x) , \qquad  \mbox{and}~ P^+ < \rho^+. 
\]
From Lemma~\ref{lm:A} we have $P^+ > \hat \rho$. 
Let $\epsilon>0$. There exist $M\in\mathbb{R}$ sufficiently large, such that
$P^+ -P(x) < \epsilon$ for every $x>M$. This implies
\[ P(x+\ell/P(x)+h) - P(x) < \epsilon\qquad \forall x>M.\]
However, using~\eqref{eq:LDP} we have 
\[
P(x+\ell/P(x)+h) - P(x)  > 
\frac{P^+-\epsilon}{L_v} \left[ \frac{1}{f(\rho^+)} - \frac{1}{f(P^+)}\right],
\]
a contradiction. We conclude that $P^+=\rho^+$. 
A similar analysis yields the asymptotic limit at $x\to-\infty$.
This proves the existence of the asymptotic value problem. 

We remark that, if $P(\cdot)$ is a solution to the asymptotic value problem, then 
any horizontal shift of $P(\cdot)$ is also a  solution.

\bigskip

\noindent\textbf{Uniqueness.}  
Assume that there exist two solutions $P_1$ and $P_2$ 
that are distinct after any horizontal shift. 
Then, we can consider some shifted versions of $P_1,P_2$ 
such that their graphs cross each other.
Let $\hat x$ be the rightmost point where they cross, and assume
\begin{equation}\label{PA}
P_1(\hat x) =P_2(\hat x),
\qquad P_1(x) > P_2(x) \quad \forall x>\hat x.
\end{equation}

Denote by $A^{P_1}, A^{P_2}$ and $L^{P_1}, L^{P_2}$ 
the averaging operators and the leader operators corresponding to $P_1,P_2$, respectively.
By the assumptions~\eqref{PA}, we have, for all $x\ge \hat x$, 
\begin{equation}\label{eq:ag}
A^{P_1} (x) > A^{P_2}(x), 
\qquad  \mbox{therefore}
\qquad \frac{1}{v\left(A^{P_1}(x)\right)} >\frac{1}{v\left(A^{P_2}(x)\right)}.
\end{equation}
Observe that we have 
\[
L^{P_1}(\hat x) =\hat x + \frac{\ell}{P_1(\hat x)} = \hat x + \frac{\ell}{P_2(\hat x)} = L^{P_2}(\hat x).
\]
Since both profiles admit the same period, we have
\[
\int_{\hat x}^{L^{P_1}(\hat x)} \frac{1}{v(A^{P_1}(z))}\; dz =
\int_{\hat x}^{L^{P_2}(\hat x)} \frac{1}{v(A^{P_2}(z))}\; dz,
\]
a contradiction to~\eqref{eq:ag}. Thus, we conclude that the solutions of the  
asymptotic value problem are unique, up to horizontal shifts. 
\end{proof}

\paragraph{Sample profiles.}
Sample profiles for $P(\cdot)$ with various $(\rho^-,\rho^+)$ values
and $w(\cdot)$ functions are given in Figure~\ref{fig:Rs}.  
The profiles are generated using the approximate solutions described in 
Theorem~\ref{th2b}. 

\begin{figure}[htbp]
\begin{center}
\includegraphics[width=7.5cm]{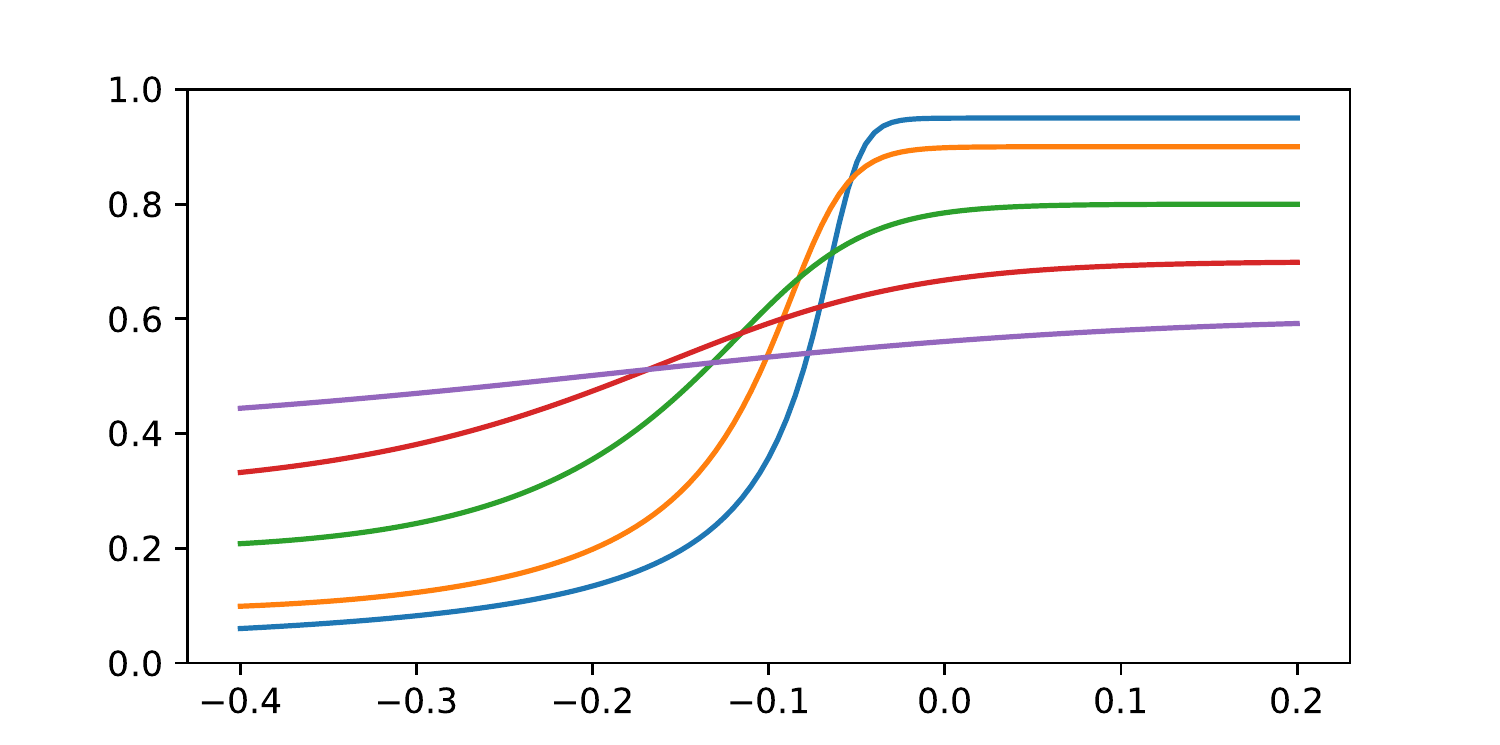}$\quad$
\includegraphics[width=7.5cm]{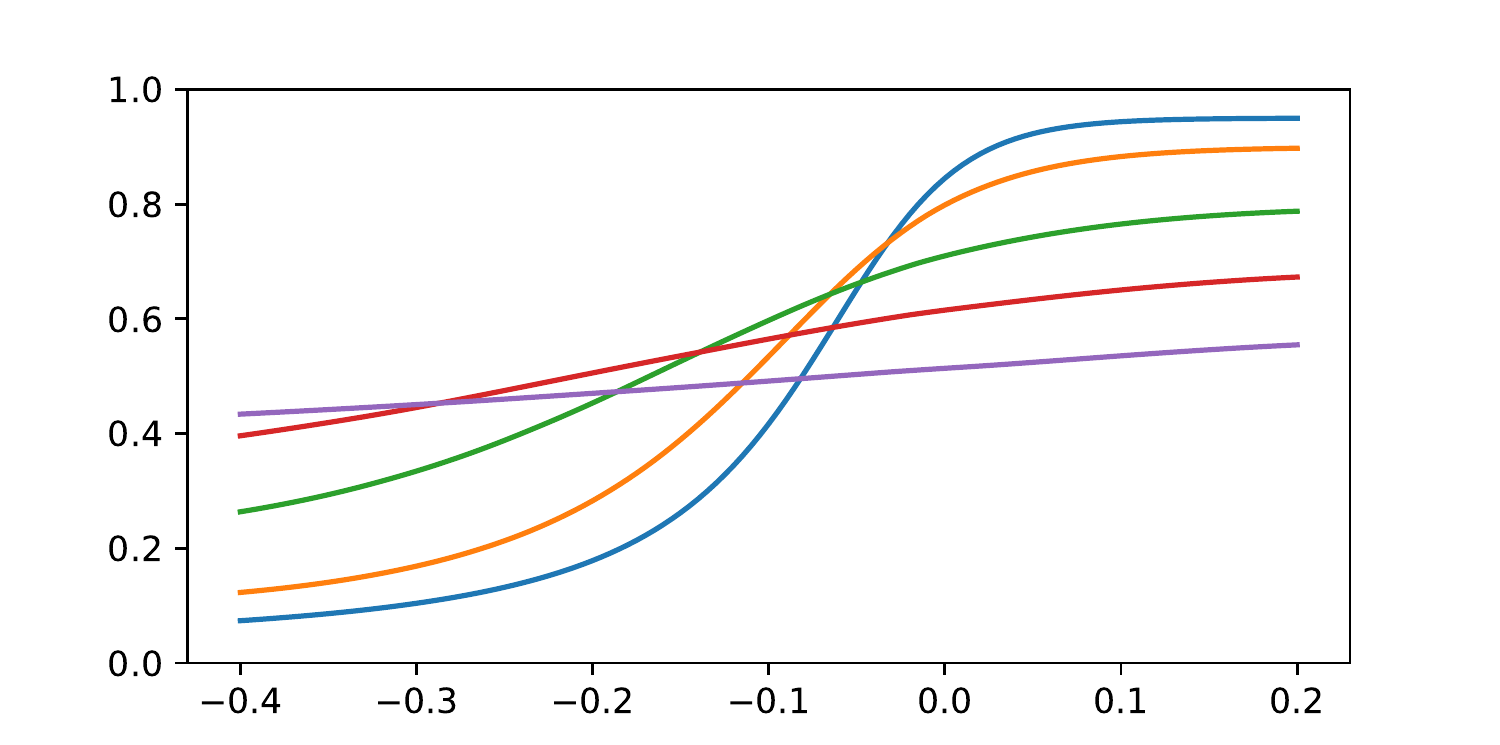}
\caption{Typical profiles $P(\cdot)$ with $v(\rho)=1-\rho, h=0.2, \ell=0.01$ 
and various $\rho^\pm$ values.
Left: $w(x)=\frac{2}{h}-\frac{2x}{h^2}$ on $(0,h)$, where $w'<0$. 
Right: $w(x) = \frac{2x}{h^2}$ on $(0,h)$, where $w'>0$. }
\label{fig:Rs}
\end{center}
\end{figure}

\subsection{Stability of the traveling waves}

It is natural to assume that the road condition right in front of the driver is 
more important than 
the condition further ahead. This leads to the additional assumption 
\begin{equation}\label{eq:wc2}
  w'(x) \le 0\quad\forall x\in(0,h).
\end{equation}
Furthermore, we also  assume  
\begin{equation}\label{eq:vc2}
 v''(\rho)\le0 \quad \forall \rho\in[0,1].
\end{equation}
This assumption gives $f''(\rho) <0$ for all $\rho\in[0,1]$, where $f(\rho)=\rho v(\rho)$. 
With these additional assumptions, 
the traveling wave profiles turn out to be local attractors for solutions of 
the FtLs model~\eqref{FtLs}.
Specifically, with mild assumptions on the initial condition,
the traveling wave profiles provide time asymptotic limits for 
the solutions of the FtLs model~\eqref{FtLs}.

\begin{theorem}\label{th3b}
Let $w(\cdot)$ satisfy~\eqref{eq:wc1} and~\eqref{eq:wc2}, and $v(\cdot)$ satisfy~\eqref{eq:vc} and~\eqref{eq:vc2}, and let  $\rho^-,\rho^+\in\mathbb{R}$ be given with
  \[
0< \rho^-<\hat\rho<\rho^+<1,\quad
f(\rho^-)=f(\rho^+) = \bar f. 
\]
Let $P(\cdot)$ be the solution for the asymptotic value problem, satisfying~\eqref{eq:P},
the asymptotic conditions~\eqref{eq:BC},
and the additional condition $P(0)=\hat\rho$.
Let $\{z_i(t): i\in\mathbb{Z}\}$ be the solution of~\eqref{FtLs} with initial condition $\{z_i(0): i\in\mathbb{Z}\}$,
and  let $\{\rho_i(t): i\in\mathbb{Z}\}$
be the corresponding discrete densities, defined in~\eqref{defrho}.

Assume that there exist two constants $c_1,c_2 \in\mathbb{R}$, 
such that the initial condition satisfies
\begin{equation}\label{P1P2}
P(z_i +c_1)  \ge \rho_i(0) \ge P(z_i+c_2), \qquad \forall i\in\mathbb{Z},
\end{equation}
Then, there exists a constant $\hat c \in\mathbb{R}$, 
such that 
\begin{equation}\label{hatc}
\lim_{t\to\infty} \left[\rho_i(t)  - P(z_i(t)+\hat c) \right] =0, \qquad \forall i\in\mathbb{Z}.
\end{equation}
\end{theorem}

\begin{proof}
\textbf{Step 1.} 
We first observe that assumption~\eqref{P1P2} implies
\[\lim_{i\to \pm\infty} \rho_i(0) = \rho^\pm. \]
Fix a time $t\ge 0$, and let $\{z_i(t):i\in\mathbb{Z}\}$ be the solution of the FtLs model 
and $\{ \rho_i(t):i\in\mathbb{Z}\}$ the corresponding discrete densities.
Denote by $\hat P(\cdot)$ the profile that satisfies 
\[ 
\hat P(x) = P(x+\tilde c) \quad \forall x\in\mathbb{R}, ~\mbox{for some}~\tilde c\in\mathbb{R}, 
\qquad \mbox{and}\quad 
\hat P(z_i(t)) \ge \rho_i(t), \quad \forall i \in\mathbb{Z}.
\]
Let $k$ be an index such that 
\begin{equation}\label{eq:w1}
\hat P(z_k(t)) =\rho_k(t), \qquad  \mbox{and}\quad 
\hat P(z_i(t)) > \rho_i(t) \quad \forall i >k,
\end{equation}
and hence
\[
L^{\hat P}(z_k) = z_{k+1}=z_k+\frac{\ell}{\hat P(z_k)}.
\]

We claim that 
\begin{equation}\label{eq:ww}
\frac{\dot\rho_k}{\dot z_k}  < \hat P'(z_k). 
\end{equation}
Indeed,~\eqref{eq:w1} and $v'\le 0$ imply
\begin{equation}\label{eq:w2}
 A^{\hat P}(z_k) > \rho^*_k, \quad A^{\hat P}\left(L^{\hat P}(z_k)\right) > \rho^*_{k+1}.
\end{equation}
Furthermore, using $v'\le 0$ we get
\begin{equation}\label{eq:w3}
v(A^{\hat P}(z_k)) \le v(\rho^*_k).  
\end{equation}

Equation~\eqref{eq:P} can be written as
\[
\hat P'(z_k) =  A_1 B_1 C_1,
\]
where
\[
A_1 \,\dot=\, \frac{\hat P(z_k)}{v(A^{\hat P}(z_k))} ,\quad 
B_1 \,\dot=\,
\frac{v(A^{\hat P}(z_k))-v(A^{\hat P}(L^{\hat P}(z_k)))}{A^{\hat P}(L^{\hat P}(z_k)) -A^{\hat P}(z_k) } ,
\quad
C_1 \,\dot=\, 
\frac{A^{\hat P}(L^{\hat P}(z_k)) -A^{\hat P}(z_k) }{L^{\hat P}(z_k)-z_k} .
\]
On the other hand, equations~\eqref{FtLs}  and~\eqref{rhodot} lead to
\[
\frac{\dot\rho_k}{\dot z_k} =  A_2 B_2 C_2,
\]
where 
\[
A_2 \,\dot=\, \frac{\rho_k}{v(\rho^*_k)},\quad 
B_2 \,\dot=\, \frac{v(\rho^*_k)-v(\rho^*_{k+1})}{ \rho^*_{k+1} - \rho^*_k}, \quad
C_2 \,\dot=\, \frac{ \rho^*_{k+1} - \rho^*_k}{z_{k+1}-z_k} .
\]
By~\eqref{eq:w1} and~\eqref{eq:w2}, we have $A_2 < A_1$. 
Since $v '' \le 0$, by~\eqref{eq:w2} we have $B_2 \le B_1$.
Finally,  to compare $C_1$ and $C_2$, 
let $\{y_i\}$ be the car distribution generated by the profile $\hat P(\cdot)$ with 
$y_k=z_k$.   We also have  $y_{k+1}=z_{k+1}$.
Define the piecewise constant functions $\hat P^\ell(\cdot)$ and $\rho^\ell(\cdot,\cdot)$ as
\begin{eqnarray*}
  \hat P^\ell(x) &\dot= &\hat P(y_i)\qquad \mbox{for}~ x\in [y_i,y_{i+1}), \\
\rho^\ell(x,t) &\dot=& \rho_i(t) \qquad \mbox{for}~ x\in [z_i(t),z_{i+1}(t)).
\end{eqnarray*}
We have 
\begin{equation}\label{eq:w4}
\hat P^\ell(x) > \rho^\ell(x,t)\qquad \forall x > z_{k+1}. 
\end{equation}
Moreover, 
\[
\rho^*_{k+1} - \rho^*_k = {\color{black}-} \int_{z_k}^{z_{k+1}} \rho_k  w(y-z_k)\; dy + 
\int_{z_{k+1}}^\infty \rho^\ell (y,t) \left[w(y-z_{k+1})-w(y-z_k)\right]\; dy,
\]
and 
\[
A^{\hat P}\left(L^{\hat P}(z_k)\right) -A^{\hat P}(z_k)  = {\color{black}-}  \int_{z_k}^{z_{k+1}} \hat P(z_k)  w(y-z_k) dy + 
\int_{z_{k+1}}^\infty \hat P^\ell (y) \left[w(y-z_{k+1})-w(y-z_k)\right]dy.
\]
Since $w'\le0$, we have $ w(y-z_{k+1})-w(y-z_k)\ge 0$.  
Using~\eqref{eq:w1} and~\eqref{eq:w4}, we conclude that $C_2\le C_1$. 
This proves~\eqref{eq:ww}. 

\medskip

On the other hand,  let $\tilde P(\cdot)$ be a profile such that 
\[  \tilde P(x) =P(x+\bar c) \quad \forall x\in\mathbb{R},~ \mbox{for some}~\bar c\in\mathbb{R}, \qquad \mbox{and} \quad 
 \tilde P(z_i(t)) \le \rho_i(t), \quad \forall i.\]
Let $k$ be an index such that 
\begin{equation}\label{eq:w11}
\tilde P(z_k(t)) =\rho_k(t), \qquad  \mbox{and}\quad 
\tilde P(z_i(t)) < \rho_i(t) \quad \forall i >k.
\end{equation}
Then, by a totally similar argument one concludes
\begin{equation}\label{eq:ww1}
\frac{\dot\rho_k}{\dot z_k}  > \tilde P'(z_k). 
\end{equation}

\medskip
\textbf{Step 2.} 
The stability of the stationary profiles 
is a consequence of~\eqref{eq:ww} and~\eqref{eq:ww1}.  
Let $P(\cdot)$ denote the profile with $P(0)=\hat\rho$, where
$\hat\rho$ is defined in~\eqref{eq:stag}. 
Since any horizontal shift of $P(\cdot)$ is also a profile, we have a family of non-intersecting
profiles generated by horizontal shifts of $P(\cdot)$. 
Then, in the $(x,P)$-plane, any point $(x,\rho)$ with $\rho^-<\rho<\rho^+$ 
must lie on a unique profile. 
This motivates the introduction of the following  mapping
\begin{equation}\label{Phi}
\Phi(x,\rho)\;\dot=\;P(0), \quad \mbox{where $P(\cdot)$ is a profile such that}\quad P(x)=\rho.
\end{equation}

Let $\{z_i(t): i\in\mathbb{Z}\}$  be the solution of~\eqref{FtLs} and $\{ \rho_i(t): i\in\mathbb{Z}\}$ the corresponding discrete densities, as in the setting of the theorem. 
Define the functions
\[
\phi_i(t) \;\dot=\;\Phi(z_i(t), \rho_i(t)), \qquad i\in \mathbb{Z}.
\]
Fix a time $t\ge 0$. 
Let $k_{\min}$ and $k_{\max}$ be the indices where $\{\phi_i(t): i\in\mathbb{Z}\}$ attains its minimum 
and maximum values, respectively,  such that
\begin{eqnarray*}
  \phi_i(t)\ge \phi_{k_{\min}}(t)\quad  \forall i\in\mathbb{Z}, ~&\mbox{and}& 
\phi_i(t)> \phi_{k_{\min}}(t) \quad \forall i>k_{\min},\\
  \phi_i(t)\le \phi_{k_{\max}}(t) \quad \forall i\in\mathbb{Z}, ~&\mbox{and}&
\phi_i(t)< \phi_{k_{\max}}(t)  \quad \forall i >k_{\max}.
\end{eqnarray*}

By the results in  Step 1, we now have 
\[
\frac{d}{dt} \phi_{k_{\min}}(t) >0, \qquad 
\frac{d}{dt} \phi_{k_{\max}}(t) <0. 
\]
This further implies that 
\[
\lim_{t\to\infty}  \left[\phi_{k_{\max}}(t) - \phi_{k_{\min}}(t) \right]= 0, 
\qquad \mbox{therefore}\quad 
\lim_{t\to\infty} \phi_i(t) = \hat\phi=\mbox{constant} \quad \forall i\in\mathbb{Z}.
\]
This proves~\eqref{hatc}, where $\hat c$ satisfies $P(-\hat c )=\hat\phi$. 
\end{proof}

\textbf{Numerical simulations.}
We consider an initial condition $\{z_i(0),\rho_i(0)\}$ which satisfies
\begin{equation}\label{SimIC}
\rho_i(0) = \begin{cases} 0.2, & z_i(0) \le -0.3,\\
0.5-0.3*\sin(5\pi z_i(0)), \qquad& -0.3 < z_i(0)< 0.3,\\
0.8, & z_i(0)\ge 0.3.
\end{cases}
\end{equation}
see the left plots in Figure~\ref{fig:Pins}.  
Typical solutions of the FtLs model $(z_i(t), \rho_i(t))$ at various time $t$
are given in the same figure, for twos different weight function $w(\cdot)$.
We observe that if $w'<0$, the oscillations damp out quickly as $t$ grows, 
and the solution approaches some profile 
$P(\cdot)$. 
On the other hand, when $w'<0$, the solution becomes more oscillatory  as $t$ grows, 
indicating the instability of the profiles.

\begin{figure}[htbp]
\begin{center}
\includegraphics[width=14cm]{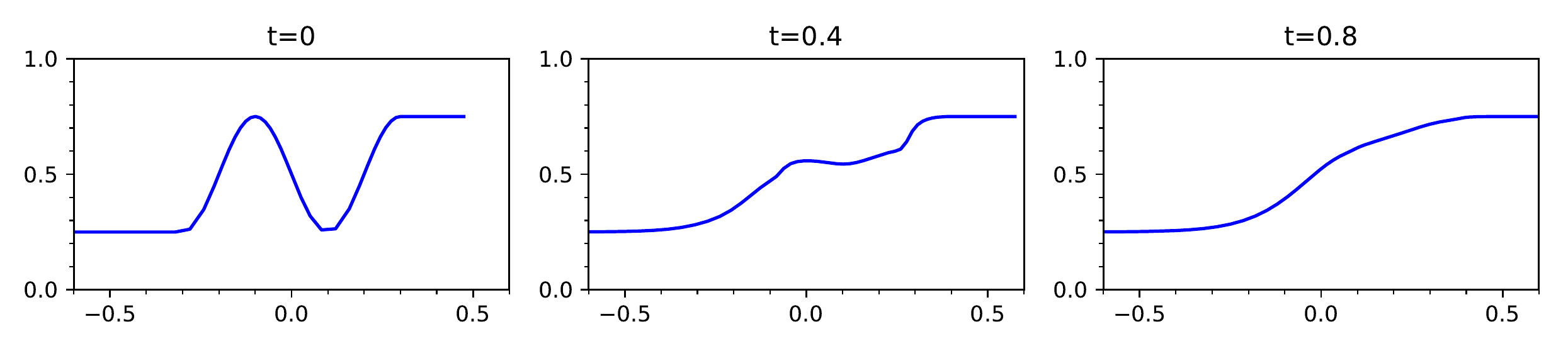}\\
\includegraphics[width=14cm]{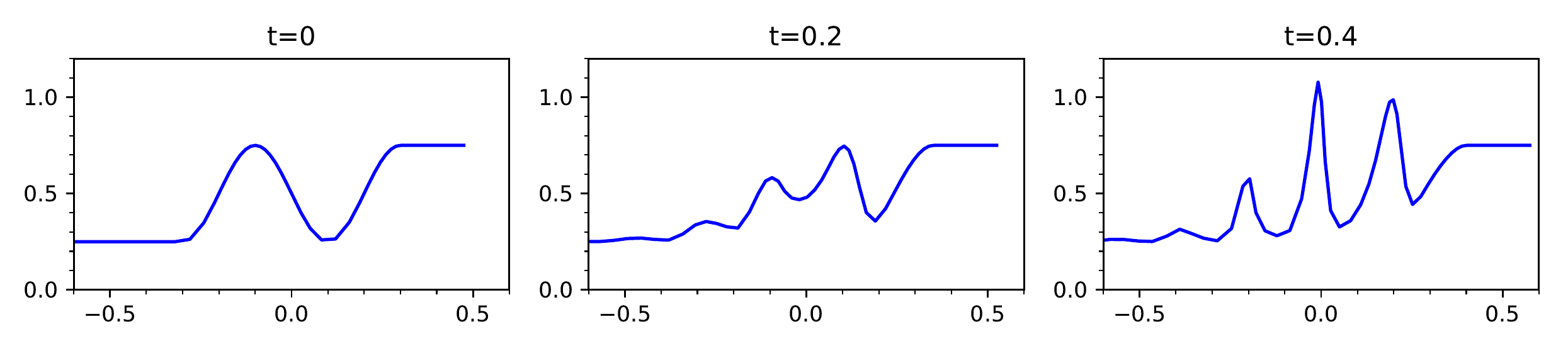}
\caption{Typical solutions of the FtLs model $(z_i(t),\rho_i(t))$ at various time $t$, with oscillatory initial condition.
Above: 
$w(x)=\frac{2}{h}-\frac{2x}{h^2}$ and the solution approaches some profile as $t$ grows.
Below: $w(x) = \frac{2x}{h^2}$ and the solution oscillates more as $t$ grows.}
\label{fig:Pins}
\end{center}
\end{figure}

\begin{remark}
If the initial condition satisfies the assumptions in Theorem~\ref{th3b}, 
then the solution approaches a stationary profile as $t\to\infty$, 
by Theorem~\ref{th3b}. 
The above numerical simulation suggests possible improvements for Theorem~\ref{th3b}.
Indeed, note that for the profile $P(\cdot)$ with $\lim_{x\to-\infty} =0.2$ and  
$\lim_{x\to\infty} =0.8$, 
one has that $0.2 < P(x) < 0.8$ for every bounded $x$. 
We remark that  the initial condition~\eqref{SimIC}  
does not satisfy the assumptions in Theorem~\ref{th3b},
since $\rho(x)=0.8$ for $x\ge 0.3$. 
Nevertheless, we observe stability in the simulation.
This indicates that the basin of attraction is probably larger than 
the assumptions in Theorem~\ref{th3b}.
\end{remark}

\section{The non-local conservation law}
\setcounter{equation}{0}
\label{sec:NLCL}

In this section 
we consider the stationary traveling wave profile $Q(\cdot)$ for~\eqref{eq:claw1}.
We denote by $\mathcal{A}$ the continuous averaging operator
\begin{equation}
\mathcal{A}(Q;x) 
\;\dot=\; \int_x^{x+h} Q(y) \,w(y-x) \; dy
=\int_0^h Q(x+s) w(s) \; ds, \label{eq:AQ}
\end{equation}
and with a slight abuse of notation,  we denote the operator also for a function 
of two variables, 
\begin{equation}
\mathcal{A}(\rho;t,x) 
\;\dot=\;
\int_x^{x+h} \rho(t,y) w(y-x) \; dy = \int_0^h \rho(t,x+s) w(s)\; ds.
\label{eq:Q*}
\end{equation}

A stationary profile for~\eqref{eq:claw1} satisfies
\begin{equation}\label{eq:p1}
Q(x) \cdot v(\mathcal{A}(Q;x) ) 
\equiv \bar f = \mbox{constant}.
\end{equation}
In the case where 
\[\lim_{x\to-\infty} Q(x) = \rho^-, \quad\mbox{and/or}  \lim_{x\to+\infty} Q(x) = \rho^+,\]
the constant $\bar f$ satisfies (respectively)
\[
\bar f = \lim_{x\to-\infty} Q(x) \cdot v(\mathcal{A}(Q;x) )
= f(\rho^-) ,
\quad \mbox{and/or} \quad 
\bar f = \lim_{x\to+\infty} Q(x) \cdot v(\mathcal{A}(Q;x) ) = f(\rho^+).
\]

We seek continuous solutions $Q(\cdot)$ for~\eqref{eq:p1}. 
Differentiating~\eqref{eq:p1} in $x$, we can
rewrite it equivalently  as a differential equation:
\begin{equation}\label{eq:p2}
Q'(x) = - \frac{Q(x)v'(\mathcal{A}(Q;x) ) }{v(\mathcal{A}(Q;x) ) }
\int_0^h Q'(x+s) w(s) \, ds.
\end{equation}
Note that~\eqref{eq:p2} is a \textbf{delay integro-differential equation}.
We remark that, for smooth solutions of $Q(\cdot)$, 
\eqref{eq:p1} and~\eqref{eq:p2} are equivalent. 
In the sequel we consider the initial value problems in Section~3.2, 
where the solution for $Q(\cdot)$ is only continuous, 
and the 
derivative $Q'(x)$  in~\eqref{eq:p2}  is assumed to be the left derivative $Q'(x-)$.

\subsection{Technical lemma}

The existence of solutions for~\eqref{eq:p2} will be established in 
Section~\ref{sec:3.2} for initial value problems, 
and in Section~\ref{sec:3.3} for asymptotic value problems. 
Assuming that the solutions exist, we establish some technical lemma.

\begin{lemma}[Asymptotic limits]\label{lm:AL}
Assume that $Q(\cdot)$ is a solution of~\eqref{eq:p2} which satisfies
\begin{equation}\label{eq:Lim}
\lim_{x\to-\infty} Q(x) = \rho^-,\qquad 
\lim_{x\to +\infty} Q(x) = \rho^+,\qquad \lim_{x\to\pm\infty} Q'(x)=0.
\end{equation}
Then, the following holds.
\begin{itemize}
\item[i)] As $x\to+\infty$, $Q(x)$ approaches $\rho^+$ with an exponential rate 
if and only if  $\rho^+ > \hat\rho$, where $\hat\rho$  satisfies~\eqref{eq:stag}. The rate $\lambda_+$ satisfies the estimate
\begin{equation}\label{estL+}
\lambda_+ >  \frac{1}{h} 
\ln\left(-\frac{\rho^+ v'(\rho^+)}{v(\rho^+)}\right).
\end{equation}
\item[ii)] As $x\to-\infty$, $Q(x)$  approaches $\rho^-$ with an exponential rate 
if and only if  $\rho^- < \hat\rho$, where $\hat\rho$  satisfies~\eqref{eq:stag}. 
The rate $\lambda^-$ satisfies the estimate
\begin{equation}\label{estL-}
\lambda_->  \frac{1}{h} 
\ln\left(-\frac{\rho^- v'(\rho^-)}{v(\rho^-)}\right).
\end{equation}
\end{itemize}
\end{lemma}

\begin{proof} 
We consider the limit $x\to+\infty$, and assume that $Q(x)\to \rho^+$ in the limit.
We linearize~\eqref{eq:p1} at $\rho^+$ and write
\[
Q(x) = \rho^+ + \eps(x),
\]
where $\eps(x)$ is a small perturbation. 
Keeping only the first order terms of $\eps$ and
using the notation ``$\approx$'', we compute, 
\begin{eqnarray*}
\mathcal{A}(Q;x) &=& \int_x^{x+h} (\rho^+ + \eps(y)) w(y-x) \; dy 
= \rho^+ + \int_x^{x+h} \eps(y) w(y-x)\; dy,
\\
v(\mathcal{A}(Q;x)) &\approx & v(\rho^+) + v'(\rho^+)\cdot \int_x^{x+h} \eps(y) w(y-x)\; dy.
\end{eqnarray*}
Putting these in~\eqref{eq:p1}, one gets
\[
(\rho^+ + \eps(x)) \cdot \left[ v(\rho^+) + v'(\rho^+)\cdot \int_x^{x+h} \eps(y) w(y-x)\; dy
\right] \approx \bar f= \rho^+ v(\rho^+).
\]

We obtain the following linear delay integro-differential equation for
for the perturbation $\eps(x)$:
\begin{equation}\label{eq:eps}
-\beta \int_x^{x+h} \eps(y) w(y-x)\; dy +\eps(x) =0, \qquad 
\mbox{where}\quad 
\beta \,\dot=\,  -\frac{\rho^+ v'(\rho^+)}{v(\rho^+)}. 
\end{equation}
Note that we have $\beta=b$, where $b$ is defined in~\eqref{lambda+e}.
We can solve~\eqref{eq:eps} using the characteristic equation.
We seek solutions of the form
\[
\eps(x) =  M e^{-\lambda x},
\]
where $M$ is an arbitrary constant (positive or negative),
and $\lambda\in\mathbb{R}$ is the exponential rate.
Plugging this into~\eqref{eq:eps}, we get
\begin{equation}\label{eq:lambda}
G(\lambda) \;\dot=\; \int_0^h e^{-\lambda s} w(s) \; ds - \frac{1}{\beta}=0.
\end{equation}
The function $G(\cdot)$ has the properties
\[
G(0) =1-\frac{1}{\beta}, \quad
\lim_{\lambda\to+\infty} G(\lambda)= -\frac{1}{\beta} <0, 
\qquad 
G'(\lambda) = -\int_0^h s e^{-\lambda s} w(s) \; ds <0 \quad \forall \lambda.
\]
Thus, $G(\lambda)=0$ has a unique positive solution if and only if 
$G(0) >0$, i.e., 
\[
\frac{1}{\beta} <1 
\qquad \Longleftrightarrow \qquad 
-\frac{v(\rho^+)}{\rho^+v'(\rho^+)} < 1
\qquad  \Longleftrightarrow \qquad 
\rho^+ > \hat \rho. 
\]
We denote this solution by $\lambda_+$.
For any given $\rho^+$, $h$, and $w(\cdot)$, an estimate for $\lambda_+$ can be obtained by observing
\[
 e^{-\lambda_+ h} = \int_0^h e^{-\lambda_+ h} w(s) \, ds < \frac{1}{\beta} ,
\]
which implies~\eqref{estL+}. 
A completely symmetric argument leads to the result in \textit{ii)} for the limit  $x\to-\infty$.
\end{proof}

\begin{remark}\label{rk3.1}
We observe two trivial cases. 
\begin{itemize}
\item[(1)] If $\rho^+ =\hat\rho$, then $\beta=1$, and we have $\lambda_+=0$. 
Similarly, if $\rho^-=\hat\rho$ then $\lambda_-=0$.  
Thus , if $\rho^-=\rho^+=\hat\rho$, 
the only profile is the constant function $Q(x)\equiv \hat\rho$.
\item[(2)] On the other hand, 
as $\rho^+ \to 1$, we have that $\beta^{-1}\to 0$, therefore $\lambda_+ \to \infty$. 
Similarly, as $\rho^-\to 0$, then $\lambda^- \to \infty$ as well. 
Thus, the only stationary traveling wave profile 
connecting $\rho^-=0,\rho^+=1$ is the unit step function,
taking the step at an arbitrary point.
\end{itemize}
\end{remark}

In the sequel  we consider the nontrivial cases where
\begin{equation}\label{eq:mm}
0<Q(x)<1 \quad \forall x\in\mathbb{R},\qquad 0<\rho^- < \hat\rho <\rho^+ <1.
\end{equation}

\subsection{Initial value problems}\label{sec:3.2}
Assume that $w(\cdot)$ satisfies~\eqref{eq:wc1} and $v(\cdot)$ satisfies~\eqref{eq:vc}. 
Fixing a point $x_0\in\mathbb{R}$, 
we consider the initial value problem of~\eqref{eq:p1},
with initial condition 
\begin{equation}\label{Qint}
Q(x)=\Phi(x)\qquad \mbox{for} \quad x\ge x_0,
\end{equation}
where $\Phi(\cdot)$ satisfies
\begin{equation}\label{eq:init}
0 < \Phi(x) <1,\qquad \Phi'(x) > 0, \qquad \mbox{for}~ x\ge x_0.
\end{equation}
We seek continuous solutions  $Q(\cdot)$ for~\eqref{eq:p1}, 
solved backward in $x$ on the interval $x\le x_0$. 
In this setting, the term $Q'(x)$ on the lefthand side of~\eqref{eq:p2} 
is assumed to be the left derivative $Q'(x-)$.

\begin{lemma}\label{lm1}
Assume that $Q(\cdot)$ is a solution of the  initial value problem of~\eqref{eq:p1}
with initial condition~\eqref{Qint}, satisfying~\eqref{eq:init}.
Then for all $x<x_0$, we have
\begin{itemize}
\item[(i)] 
$Q$ is positive and monotone
\begin{equation}\label{eq:LL1}
Q'(x) > 0, \qquad 0 < Q(x) < \Psi(x_0) <1.
\end{equation}

\item[(ii)]
We have
\begin{equation}\label{eq:LL2}
Q(x) v(\mathcal{A}(Q;x))=\bar f_0, \qquad \mbox{where}
\quad \bar f_0 \;\dot=\; \Phi(x_0) \cdot v\left( \mathcal{A}(\Phi;x_0) \right).
\end{equation}

\item[(iii)] 
We have the asymptotic limit
\begin{equation}\label{eq:LL3}
\lim_{x\to-\infty} Q(x) = \rho^-_0, \qquad \mbox{where} \quad
0<\rho^-_0<\hat\rho, \quad f(\rho^-_0)=\bar f_0.
\end{equation}
\end{itemize}
\end{lemma}

\begin{proof}
(i) 
We first observe that $Q'(x_0-) > 0$. 
Indeed, this follows immediately from the assumptions~\eqref{eq:init} on $\Phi(\cdot)$ and 
the equation~\eqref{eq:p2}. 

We now assume that $Q(\cdot)$ is not monotone increasing on $x<x_0$.
Let $\hat y < x_0 $ be the local minimum such that $Q'(\hat y)=0$ and 
$Q'(x) > 0$ for all $x\ge \hat y$. 
By~\eqref{eq:p2}, this implies $Q'(\hat y) >0$, a contradiction. 
Thus, we conclude that $Q'(x) >0$ for all $x<x_0$.
 
The positivity of the solution  follows from the fact that~\eqref{eq:p1}
is autonomous and $Q=0$ is a critical value.  

(ii) By~\eqref{eq:p1}, we immediately have~\eqref{eq:LL2}.

(iii)  By Lemma~\ref{lm:AL}, the asymptotic value satisfies $\rho^-_0<\hat\rho$.
Taking the limit $x\to-\infty$ in~\eqref{eq:LL2} we obtain~\eqref{eq:LL3}. 
\end{proof}

\begin{theorem}\label{th1}
Assume that $w(\cdot)$ satisfies~\eqref{eq:wc1} and $v(\cdot)$ satisfies~\eqref{eq:vc}. 
Consider the initial value problem of~\eqref{eq:p1} with initial condition~\eqref{Qint},
satisfying~\eqref{eq:init}.
Then there exists a unique solution $Q$ on the interval $x\le x_0$.
The solution $Q$ is monotone increasing and 
Lipschitz continuous, with Lipschitz constant $\bar f_0  L_v \kappa$, 
where $L_v$ is the Lipschitz constant for the map $(1/v)$
and $\kappa=\| w(\cdot)\|_\infty$. 
\end{theorem}

\begin{proof} 
\textbf{Existence.} For delay differential equations with strictly positive delays, the 
existence and uniqueness of solutions 
can be proved by the standard method of steps, cf~\cite{MR0141863, MR0477368}. 
Unfortunately, for~\eqref{eq:p2} the delay is arbitrarily small, and
the method of steps does not apply.
Instead, we apply a fixed point argument. 

Let $\Phi$ be the initial condition on $x\ge x_0$, and let
$\bar f_0$ and $\rho^-_0$ be defined as in~\eqref{eq:LL2}-\eqref{eq:LL3} in 
Lemma~\ref{lm1}.  
Let $L_v$ be the Lipschitz constant for the map $(1/v)$, and consider the constants
\begin{equation}\label{CCC}
\kappa \;\dot=\; \| w(\cdot)\|_\infty,\qquad 
L_Q \;\dot=\; \bar f_0 L_v \kappa, \qquad \gamma\;\dot=\; 2 \bar f_0 L_v \kappa.
\end{equation}
We consider the set $\mathcal{U}$ of functions, defined on $x\le x_0$, as
\begin{eqnarray}
\mathcal{U} &\dot=& \left\{  u : (-\infty,x_0]\mapsto (\rho^-_0, \Phi(x_0)]~;~
u ~\mbox{is Lipschitz with Lipschitz constant} ~L_Q ,
\right. 
\nonumber \\
&& \qquad \left. u(x_0) = \Phi(x_0), ~\lim_{x\to-\infty}u(x)=\rho^-_0, ~u'(x) >0 ~\forall x\le x_0.
\right\}
\label{def:U}
\end{eqnarray}

Let $u\in\mathcal{U}$. 
We define a Picard operator on $\mathcal{U}$ as
\bel{Picard}
(\P u)(x)~\doteq~  \frac{\bar f_0}{v(\mathcal{A}(u;x))}, 
\qquad\forall x\geq x_0.
\eeq
Note that a fixed point for $\P $ is a solution for~\eqref{eq:p1}.

We first claim that the Picard operator $\P$ maps $\mathcal{U}$ into itself, i.e.
\begin{equation}\label{P-inv}
(\P u) \in \mathcal{U} \quad \mbox{if} ~ u \in \mathcal{U} .
\end{equation}
Indeed, from the construction we have 
\[(\P u) (x_0) = \frac{\bar f_0}{ v(\mathcal{A}(\Phi;x_0))} = \Phi(x_0).\]
Moreover, since $v$ is decreasing, we conclude that $(1/v)$ is increasing. 
Then, since $u\in\mathcal{U}$ is monotone increasing, 
so is the averaged function $\mathcal{A}(u;x)$. 
Therefore $(\P u)$ is also monotone increasing. 
Furthermore, for the asymptotic value, we have
\[
\lim_{x\to-\infty} (\P u)(x) = \lim_{x\to-\infty} \frac{\bar f_0}{v(\mathcal{A}(u;x))}
= \frac{\bar f_0}{v(\rho^-_0)}  = \rho^-_0.
\]
Finally, since $u$ is Lipschitz, so is $(\P u)$. To obtain the Lipschitz constant, we compute
\[
\mathcal{A}(u;x)_x = \int_0^h u'(x+s) w(s) \; ds 
\le \kappa \int_0^h u'(x+s) \; ds = \kappa \left[u(x+h)-u(x) \right] \le \kappa,
\]
therefore
\[
(\P u)'(x) = \bar f_0 \cdot \left(\frac{1}{v(\mathcal{A}(u;x))}\right)_x \le \bar f_0 L_v \mathcal{A}(u;x)_x\le \bar f_0 L_v \kappa = L_Q.
\]
We conclude that $(\P u) \in\mathcal{U}$, proving the claim~\eqref{P-inv}.

\medskip

We further claim that the Picard operator $\P$ is a strict contraction w.r.t.~the norm
\bel{norm} 
\|u\|_\gamma~\doteq~\sup_{x\leq x_0} e^{\gamma x} |u(x)|,\eeq
where $\gamma$ is defined in~\eqref{CCC}.

Indeed, let $u_1, u_2\in \mathcal{U}$. Assume
\[
\|u_1-u_2\|_\gamma~=~\delta,
\qquad \mbox{i.e.}\quad 
|u_1(x) - u_2(x)|~\leq~\delta\,e^{-\gamma x}\quad\forall x\leq x_0\,.
\]
Then, for any $x\leq x_0$ we have 
\begin{eqnarray*}
&& \hspace{-15mm}  \Big|(\P u_1)(x)-(\P u_2)(x)\Big|
~\leq~
\bar f_0 L_v\cdot \int_x^{x+h}
| u_1(y)-u_2(y)|\, w(y-x)\, dy
\\
&\leq& \bar f_0 L_v\kappa \cdot \int_x^{x+h}\delta\,e^{-\gamma y}\, dy
~\leq~\bar f_0 L_v\kappa \delta {1\over \gamma} \Big( e^{-\gamma x} - e^{-\gamma(x+h)}\Big)~
<~{\delta\over 2} e^{-\gamma x}.
\end{eqnarray*}
Hence
\[
\Big\| (\P u_1)-(\P u_2)\Big\|_\gamma~\leq~\sup_{x\leq x_0} e^{\gamma x}
\Big|(\P u_1)(x)-(\P u_2)(x)\Big|~\leq~{\delta\over 2}~=~{1\over 2} \|u_1-u_2\|_\gamma\,.
\]
This shows that the Picard operator is a strict contraction from $\mathcal{U}$ to itself,
hence it has a unique fixed point in $\mathcal{U}$. 
The fixed point iterations converge pointwise on bounded sets. 
Furthermore, since all functions in $\mathcal{U}$ 
have a fixed asymptotic limit as $x\to-\infty$, 
we conclude the pointwise convergence for all $x\le x_0$.
This establishes the existence of solutions for the initial value problem.

\bigskip

\noindent\textbf{Uniqueness.} 
The uniqueness of solutions can be proved by contradiction. 
 Let $Q(\cdot)$ and $\widehat Q(\cdot)$ be two distinct 
solutions of the initial value problem, 
with the same initial condition~\eqref{Qint}. 
Without loss of generality, we assume that for some $\bar x\le x_0$ we have 
$Q(x) = \widehat Q(x)$ on $x\ge \bar x$ and 
$Q(x) > \widehat Q(x)$ on 
some non-empty interval $[\bar x -c, \bar x]$ where $c\in\mathbb{R}^+$. 
Since the solutions are monotone, there exist $x_1,x_2$, with 
$\bar x -c < x_1<x_2<\bar x < x_1+h$ and 
\begin{equation}\label{ASN}
Q(x_1)=\widehat Q(x_2), \quad Q(x)>\widehat Q(x) \quad x\in[x_1,\bar x),
\qquad
\mbox{and}
\quad 
Q'(x) < \widehat Q'(x) \quad x\in[x_1,\bar x],
\end{equation}
see Figure~\ref{fig:1} for an illustration.  Note that,  since $Q(\cdot)$ and $\widehat Q(\cdot)$ 
are  smooth functions, by continuity the assumptions~\eqref{ASN} 
hold for some $x_1,x_2$ 
sufficiently close to $\bar x$.
 This  implies
\[
Q(x_1 + s) < \widehat Q(x_2+s),\qquad \forall s\in(0,h].
\]

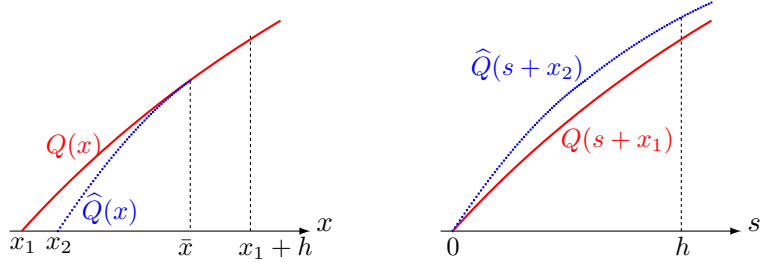
\begin{figure}[htbp]
\begin{center}
\setlength{\unitlength}{0.8mm}
\begin{picture}(70,40)(0,-4)  
\put(0,0){\vector(1,0){50}}\put(51,0){$x$}
\multiput(30,25)(0,-1){25}{\line(0,-1){0.5}}
\multiput(40,0)(0,1){33}{\line(0,1){0.5}}
\put(0,-3){\small$x_1$} 
\put(6,-3){\small$x_2$} 
\put(28,-4){\small$\bar x$} 
\put(38,-4){\small$x_1+h$} 
\thicklines
\color{red}
\qbezier(2,0)(20,20)(45,35)
\put(6,13){\small$Q(x)$}
\color{blue}
\qbezier[60](8,0)(20,18)(30,25)
\put(12,2){\small$\widehat Q(x)$}
\end{picture}
\begin{picture}(60,36)(0,-4)  
\put(0,0){\vector(1,0){50}}\put(51,0){$s$}
\multiput(40,0)(0,1){36}{\line(0,1){0.5}}
\put(1,-4){\small $0$}
\put(39,-4){\small $h$}
\thicklines
\color{red}
\qbezier(2,0)(20,20)(45,35)
\put(20,14){\small$Q(s+x_1)$}
\color{blue}
\qbezier[60](2,0)(14,18)(24,25)
\qbezier[50](24,25)(34,33)(45,38)
\put(5,26){\small$\widehat Q(s+x_2)$}
\end{picture}
\caption{Left: Graphs of $Q(x)$ and $\widehat Q(x)$ on $[x_1, x_1+h]$. 
Right: Graphs of shifted functions $Q(s+x_1)$ and $\widehat Q(s+x_2)$.}
\label{fig:1}
\end{center}
\end{figure}

We now have
\begin{equation}\label{eq:c2}
\mathcal{A}(Q;x_1) < \mathcal{A}(\widehat Q;x_2).
\end{equation}
Since both $Q(\cdot),\widehat Q(\cdot)$ are solutions of~\eqref{eq:p1}, we have
\[
Q(x_1) v(\mathcal{A}(Q;x_1)) = \widehat Q(x_2) v(\mathcal{A}(\widehat Q;x_2))
\,,\quad \text{and hence}\quad
v(\mathcal{A}(Q;x_1))=v(\mathcal{A}(\widehat Q;x_2)),
\]
a contradiction to~\eqref{eq:c2}.
Thus, we conclude that $Q(x) \equiv \widehat Q(x)$ for all $x\le x_0$,
proving the uniqueness of the solutions for the initial value problem.
\end{proof}

\begin{remark}\label{rm:Num}
To generate profiles of $Q(x)$, the fixed point iterations for the Picard operator $\P$ is not conveniet. Instead,  we adopt the following numerical scheme: 
Fix a step size $\dx$, and discretize the space by
\[
x_i = x_0+ i \dx, \quad i\in\mathbb{Z}^-.
\]
We construct a continuous  and piecewise affine approximate solution 
$Q^\Delta(\cdot)$.  Denote the approximate value at the grid points as
\[
 Q_0 \;\dot=\; \Phi(x_0) <1, \qquad 
 \mbox{and}\quad Q_i \;\dot=\; Q^\Delta(x_i),
\quad i\in\mathbb{Z}^-.
\]
Then we have
\[
Q^\Delta(x) = Q_{i-1} \frac{x_i-x}{\dx} 
+ Q_i \frac{x-x_{i-1}}{\dx} \quad \mbox{for}~ x\in[x_{i-1},x_i],
\qquad \forall i\in\mathbb{Z}^-.
\]

Fix an $i\in\mathbb{Z}^-$, and assume that $Q^\Delta(x)$ is given for all $x\ge x_i$. The value $Q_{i-1}$ is generated by solving the nonlinear equation 
\begin{equation}\label{eq:NN}
\mathcal{G}(Q_{i-1}) \;\dot=\;  Q_{i-1} \cdot v(\mathcal{A}(Q^\Delta;x_{i-1})) -  Q_{i} 
\cdot v(\mathcal{A}(Q^\Delta;x_{i})) =0.
\end{equation}
Numerically,~\eqref{eq:NN} can be computed efficiently using  Newton iterations,
 with $Q_i$ as the initial guess.
We  remark that~\eqref{eq:NN} can be viewed as a finite difference approximation
for~\eqref{eq:p2},
\begin{equation}\label{eq:NN2}
\frac{Q_{i}-Q_{i-1}}{\dx} \cdot v(\mathcal{A}(Q^\Delta;x_{i-1})) + Q_i \cdot
\frac{v(\mathcal{A}(Q^\Delta;x_{i}))-v(\mathcal{A}(Q^\Delta;x_{i-1}))}{\dx}=0.
\end{equation}
The algorithm~\eqref{eq:NN2} is somewhat similar to the symplectic method for 
systems of ODEs.

We compute
\begin{eqnarray}
\mathcal{G}(Q_i) &=& Q_{i} \cdot \left[v(\mathcal{A}(Q^\Delta;x_{i-1})) -  v(\mathcal{A}(Q^\Delta;x_{i}))\right] > 0, \label{eq:gg0} \\
\mathcal{G}(0) &=& - Q_{i} \cdot v(\mathcal{A}(Q^\Delta;x_{i})) <0,\label{eq:gg1}
\end{eqnarray}
and 
\begin{eqnarray*}
&& \hspace{-1.5cm} \frac{\partial}{\partial Q_{i-1}} 
\mathcal{A}\left(Q^\Delta;x_{i-1}\right) 
~=~ 
\frac{\partial}{\partial Q_{i-1}}  \int_{x_{i-1}}^{x_i} \left[
Q_{i-1} \frac{x_i-y}{\dx} + Q_i \frac{y-x_{i-1}}{\dx}
\right] w(y-x_{i-1})\; dy \\
&=& \int_{x_{i-1}}^{x_i}  \frac{x_i-y}{\dx}  w(y-x_{i-1})
\; dy  ~=~ \int_0^{\dx} \frac{\dx-s}{\dx} w(s) \, ds.
\end{eqnarray*}
Then, for $\dx$ sufficiently small, the above term arbitrarily small, and we have
\begin{equation}\label{eq:gg3}
\mathcal{G}'(Q_{i-1}) = v (\mathcal{A}(Q^\Delta;x_{i-1})) 
+ Q_{i-1}  v'(\mathcal{A}(Q^\Delta;x_{i-1})) 
\cdot \int_0^{\dx} \frac{\dx-s}{\dx} w(s) \, ds
>0.
\end{equation}
By~\eqref{eq:gg0}-\eqref{eq:gg3} we conclude that there exists a unique solution  
$Q_{i-1}$ of~\eqref{eq:NN}, satisfying $0<Q_{i-1}<Q_i$. 

Iterating the above step for $i\in\mathbb{Z}^-$, we
generate a sequence of solutions $\{Q_i\}$ satisfying 
\begin{equation}\label{eq:Fi}
0 < Q_{i-1} <Q_i< Q_0, \quad   Q(x_i) v(\mathcal{A}(Q^\Delta;x_i) )= Q(x_0) v(\mathcal{A}(Q^\Delta;x_0))=\bar f_0, \quad \forall i \in\mathbb{Z}^- .
 \end{equation}

We  now establish an upper bound on the  gradient of 
$Q^\Delta$, on $x<x_0$. 
Recall the constants $L_v,\kappa,L_Q$ defined in~\eqref{CCC}. 
Using the scheme \eqref{eq:NN2}, we compute
\begin{eqnarray*}
\frac{Q_i-Q_{i-1}}{\Delta x} &=&  Q_{i-1} v(\mathcal{A}(Q^\Delta;x_{i-1}))
\left[ \frac{1}{v(\mathcal{A}(Q^\Delta;x_{i}))} - \frac{1}{v(\mathcal{A}(Q^\Delta;x_{i-1}))} \right] \\
&\le& \bar f_0 L_v \left[ \mathcal{A}(Q^\Delta;x_{i}) -\mathcal{A}(Q^\Delta;x_{i-1}) \right] \le \bar f_0 L_v  \kappa  = L_Q.
\end{eqnarray*}
Since $Q^\Delta(\cdot)$ is piecewise affine on $x<x_0$, we conclude that
\[ (Q^\Delta)'(x) \le L_Q \quad \forall x<x_0.\]
Convergence of  the sequence $\{Q^{\dx}(\cdot)\}$ as $\dx\to 0$
follows from Helly's compactness theorem.
This offers an alternative proof for
the existence of solutions for the initial value problem.
\end{remark}

\subsection{Asymptotic value problems} \label{sec:3.3}

\begin{theorem}[Asymptotic value problem]\label{th2}
Assume that $w(\cdot)$ satisfies~\eqref{eq:wc1} and $v(\cdot)$ satisfies~\eqref{eq:vc}. 
Let $\rho^-\in\mathbb{R}$ and $\rho^-\in\mathbb{R}$ be given which satisfy
\[
0< \rho^-<\hat\rho<\rho^+<1 \quad\mbox{and}\quad
f(\rho^-)=f(\rho^+).
\]
Consider the asymptotic value problem for~\eqref{eq:p1}, with 
\begin{equation}\label{eq:BC2}
\lim_{x\to\infty} Q(x) = \rho^+,
\qquad
\lim_{x\to-\infty} Q(x) = \rho^-.
\end{equation}
There exist Lipschitz continuous and monotone solutions for 
the asymptotic value problem. 

Furthermore, the solutions are unique up to horizontal shifts, in the following sense. 
Let  $Q_1(\cdot)$ and $Q_2(\cdot)$ be two solutions of the asymptotic value problem
with the same asymptotic conditions~\eqref{eq:BC2}, then 
there exists a constant $c\in\mathbb{R}$ such that
$Q_1(x)=Q_2(x+c)$ for all $x\in\mathbb{R}$.
\end{theorem}

\begin{proof}
\textbf{Existence.} 
The  existence of solutions to the asymptotic value problem 
is established through convergence of approximate solutions,
similar to the approach used for the proof of Theorem~\ref{th2b}.
Let $\lambda_+>0$ be the exponential rate given in Lemma~\ref{lm:AL},
and let $\{x_n: n\in\mathbb{N}\}$ be an increasing sequence of real numbers 
such that 
 $ \lim_{n\to\infty} x_n = +\infty$. 
For each given $n\in\mathbb{N}$,  let $Q^{(n)}(\cdot)$ 
be the solution of the initial value problem of~\eqref{eq:p1}, defined on $x\le x_n$,
with initial condition
\[
Q^{(n)}(x)=\Phi(x)\;\dot=\;  \rho^+ - e^{-\lambda_+ x},\quad \mbox{on}\quad x \ge x_n.
\]
By Theorem~\ref{th1}, $Q^{(n)}(\cdot)$ exists and is unique, and it satisfies
\[
Q^{(n)}(x) \cdot v(\mathcal{A}(Q^{(n)};x)) = \bar f_n 
\quad
\mbox{where}\quad
\bar f_n = \Phi(x_n) \cdot v\left(\int_{x_n}^{x_n+h} \Phi(y) w(y-x_n)\, dy\right).
\]
By Lemma~\ref{lm1}, we have 
\[
\lim_{x\to-\infty} Q^{(n)}(x) = \rho^-_n\qquad  \mbox{where} \quad
f(\rho^-_n) =\bar f_n, \quad \rho^-_n<\hat\rho.
\]
Using the exact expression of $\Phi(x)$, we compute
\begin{eqnarray*}
\lim_{n\to\infty} \bar f_n &=& \lim_{x_n\to\infty} (\rho^+ - e^{-\lambda_+ x_n}) \cdot v\left( 
\int_{x_n}^{x_n+h} (\rho^+ - e^{-\lambda_+ y})w(y-x_n)\, dy
\right)\\
&=& \rho^+ \cdot v\left(\rho^+  \int_0^h  w(s) \, ds \right)
~=~ \rho^+ v(\rho^+) = f(\rho^+)=f(\rho^-).
\end{eqnarray*}
This  implies that 
\[\lim_{n\to\infty} \rho^-_n=\rho^-.\]

Furthermore, replacing $Q$ by $Q^{(n)}$ in ~\eqref{eq:p1} and 
subtracting it  from the identity
\[
\mathcal{A}\left(Q^{(n)};x\right) \cdot v\left(\mathcal{A}(Q^{(n)};x)\right) 
= f\left(\mathcal{A}(Q^{(n)};x)\right),
\]
we get the estimate
\[
\frac{f(\mathcal{A}(Q^{(n)};x))-\bar f}{v(\mathcal{A}(Q^{(n)};x))} 
= \mathcal{A}(Q^{(n)};x) -Q^{(n)}(x) 
< Q^{(n)}(x+h) -Q^{(n)}(x).
\]

Finally, 
let $\tilde Q^{(n)}$ be a horizontally shifted function of $Q^{(n)}$, such that 
for some $c_n\in\mathbb{R}$,
\[
\tilde Q^{(n)}(x) = Q^{(n)}(x+c_n) \quad \forall x, \qquad \mbox{and}\quad
\tilde Q^{(n)}(0)=\hat \rho.
\]

Then, $\tilde Q^{(n)}$ is bounded, Lipschitz continuous, and monotonically increasing. 
 By Helly's compactness Theorem,  as $n\to\infty$, 
 there exists a subsequence 
of functions $\{\tilde Q^{(n)}\}$ that converges uniformly  on bounded set
to a limit function $Q$. 
The limit function is bounded, Lipschitz continuous, and monotone increasing, 
satisfying the integral equation~\eqref{eq:p1} and the estimate
\begin{equation}\label{eq:FEE}
\frac{f(\mathcal{A}(Q^;x))-\bar f}{v(\mathcal{A}(Q^;x))} 
 < Q(x+h) -Q(x).
\end{equation}

It remains to establish the asymptotic values of the limit function $Q$.
Since $Q$ is monotone and bounded, the limits as $x\to\pm\infty$ exist. 
Let $Q^+ \;\dot=\; \lim_{x\to+\infty} Q(x)$
and assume that $Q^+ \not=\rho^+$.  
From Lemma~\ref{lm:AL} it holds $Q^+>\hat\rho$,
and from the construction $Q^+ <\rho^+$. 
Therefore we have $\bar f > f(Q^+)$. 
Let $\epsilon>0$. There exists an $M\in\mathbb{R}$, sufficiently large, such that 
$Q^+-Q(x) <\epsilon$ for all $x>M$. 
In particular, we
have
\[  Q(x+h) - Q(x) < \epsilon \qquad \forall x >M.\]
However, from~\eqref{eq:FEE} we get, for any $x>M$,
\[
  Q(x+h)-Q(x) >  f(Q^+) - \bar f>0 
\]
a contradiction. 
We conclude that $Q^+=\rho^+$. 
A completely similar argument gives the limit as $x\to-\infty$, 
proving the asymptotic values~\eqref{eq:BC2}.
This establishes the existence of solutions for the asymptotic value problems.

\medskip

\textbf{Uniqueness.} 
The uniqueness of solutions is proved by a contradiction argument. 
Let $Q_1$ and $Q_2$ be two distinct solutions for the asymptotic value problem
with the same asymptotic values~\eqref{eq:BC2}. 
We may horizontally shift the profiles, such that the graphs of $Q_1$ and $Q_2$ intersect.
Let $y\in\mathbb{R}$ be the rightmost intersection point, such that
\[
Q_1(y)=Q_2(y), \qquad \mbox{and}\quad Q_1(x) > Q_2(x) \quad \forall x>y.
\]
Then, we have 
$\mathcal{A}(Q_1;y) > \mathcal{A}(Q_2;y)$, 
so
\begin{equation}\label{CDCD}
Q_1(y) v( \mathcal{A}(Q_1;y)) < Q_2(y) v( \mathcal{A}(Q_2;y)) . 
\end{equation}
On the other hand, by~\eqref{eq:p1} we must have 
\[
Q_1(y) v( \mathcal{A}(Q_1;y)) = Q_2(y) v( \mathcal{A}(Q_2;y)) =\bar f = f(\rho^\pm),
\]
which leads to a contradiction to~\eqref{CDCD}.
This proves that solutions to the asymptotic value problems are unique up to horizontal
shifts.
\end{proof}

\paragraph{Sample profiles.}
Sample profiles for $Q(\cdot)$ with various $(\rho^-,\rho^+)$ values
and two different weight functions $w(\cdot)$ are given in Figure~\ref{fig:2}. 
The profiles are generated using the algorithm in Remark~\ref{rm:Num}
in the approximating procedure in the proof of Theorem~\ref{th2}.

\begin{figure}[htbp]
\begin{center}
\includegraphics[width=7.5cm]{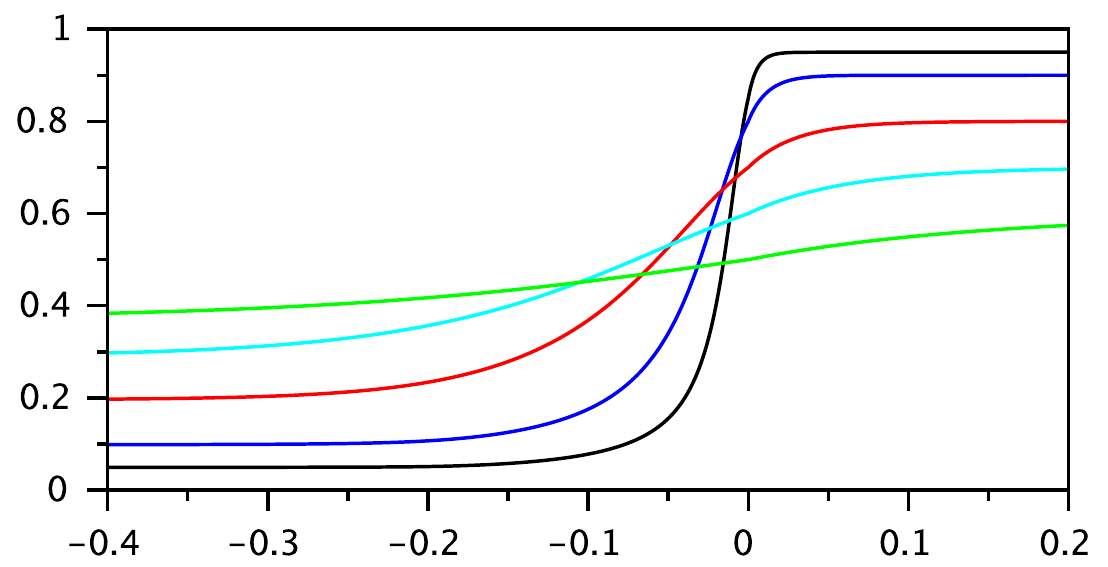}$\quad$
\includegraphics[width=7.5cm]{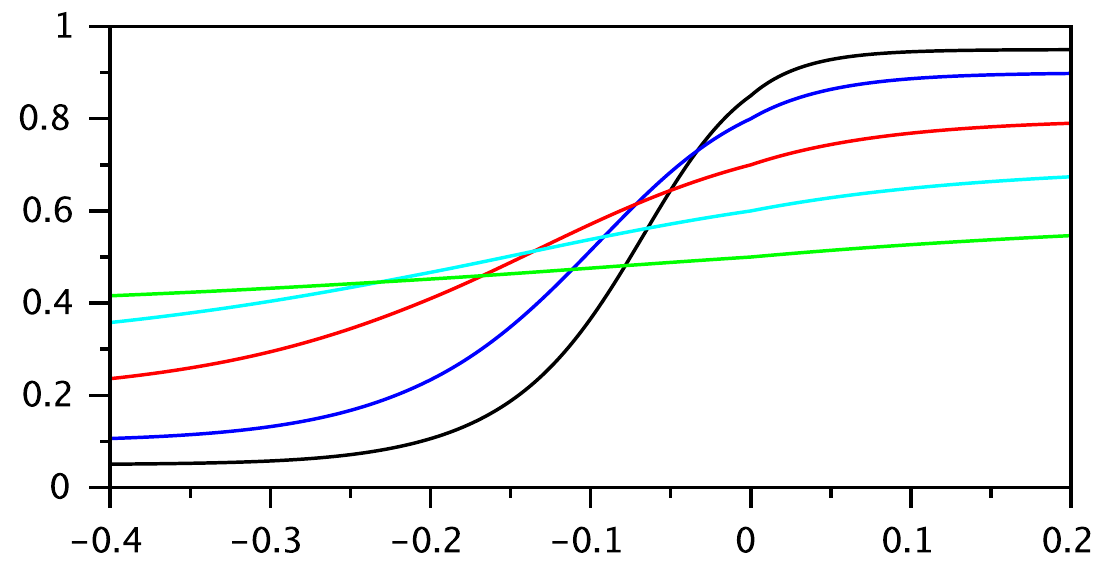}
\caption{Typical profiles $Q(\cdot)$ with $v(\rho)=1-\rho, h=0.2$ and 
various $\rho^\pm$ values.
For the left plot we use $w(x)=\frac{2}{h}-\frac{2x}{h^2}$ for $x\in(0,h)$,
and for right plot we use $w(x) = \frac{2x}{h^2}$ for $x\in(0,h)$. }
\label{fig:2}
\end{center}
\end{figure}

\subsection{Stability of the traveling waves}

For the Cauchy problem of~\eqref{eq:claw1}, the existence and uniqueness 
of entropy weak solution is established in~\cite{BG2016}. 
In particular,  if the initial condition $\rho(0,\cdot)$ is smooth, then the solution
$\rho(t,\cdot)$ remains smooth for all $t\ge 0$.

Under the additional assumptions~\eqref{eq:wc2} and~\eqref{eq:vc2}, we now show that 
the stationary profiles are the stable  time asymptotic limit for the solutions of the
Cauchy problem of the non-local conservation law, under mild assumptions on 
smooth initial condition.

\begin{theorem}[Stability]\label{th3}
Let  $w(\cdot)$ satisfy~\eqref{eq:wc1} and~\eqref{eq:wc2}, 
and let $v(\cdot)$ satisfy~\eqref{eq:vc} and~\eqref{eq:vc2}.
Let $\rho^-,\rho^+$ satisfy 
\[
f(\rho^-)=f(\rho^+) = \bar f, \qquad 0 < \rho^- < \hat\rho < \rho^+ <1.
\]
Let $Q(x)$ be the unique stationary profile with  
asymptotic conditions~\eqref{eq:BC2} and $Q(0)=\hat\rho$. 

Let $\rho(0,\cdot)$ be a smooth function, and assume that there exist 
constants $c_1\in\mathbb{R},c_2\in\mathbb{R}$ 
such that 
\begin{equation}\label{Q1Q2}
Q(x+c_1) \le \rho(0,x) \le Q(x+c_2), \qquad \forall x \in\mathbb{R}.
\end{equation}
For $t\ge0$, let $\rho(t,\cdot)$ be the solution of 
the Cauchy problem for~\eqref{eq:claw1}, 
with initial condition $\rho(0,\cdot)$. 
Then, there exists a constant $\bar c\in\mathbb{R}$ such that 
\begin{equation}\label{Q-conv}
\lim_{t\to\infty} \left[ \rho(t,x) - Q(x+\bar c) \right]=0, \qquad \forall x\in\mathbb{R}. 
\end{equation}
\end{theorem}

\begin{proof} \textbf{Step 1.}
We observe that assumption~\eqref{Q1Q2} implies 
\[ \lim_{x\to\infty} \rho(0,x) = \rho^+,
\qquad
\lim_{x\to-\infty} \rho(0,x) = \rho^-.\]
Fix a time $t\ge0$. Let $\hat Q(x) = Q(x+\hat c)$ for some $\hat c\in\mathbb{R}$ 
be a profile  such that
\[
\hat Q(x) \ge \rho(t,x), \qquad \forall x\,,
\]
and let $\hat x$ be the point where 
\begin{equation}\label{eq:s2}
\hat Q(\hat x) = \rho(t,\hat x), 
\quad
\hat Q_x(\hat x) = \rho_x(t,\hat x), 
\quad \mbox{and}\quad
\hat Q(x) > \rho(t,x) \quad \forall x>\hat x.
\end{equation}

We claim that 
\begin{equation}\label{eq:s1}
\rho_t(t,\hat x) <0, 
\qquad \mbox{i.e.,}\quad 
[\rho(t,\hat x) v(\mathcal{A}(\rho; t,\hat x))]_x >0. 
\end{equation}
Indeed,  we have the estimate
\begin{equation}\label{eq:s3}
\mathcal{A}(\hat Q;\hat x) - \mathcal{A}(\rho;t,\hat x)=
\int_0^h [\hat Q(\hat x+s)-\rho(\hat x+s)]w(s)\; ds >0. 
\end{equation}
Since $v'< 0$ and  $v''\le 0$, we have
\begin{equation}\label{eq:s4}
v(\mathcal{A}(\hat Q;\hat x)) < v(\mathcal{A}(\rho;t,\hat x)),\qquad
v'(\mathcal{A}(\hat Q;\hat x)) \le v'(\mathcal{A}(\rho;t,\hat x)).
\end{equation}
Finally, since 
\begin{eqnarray*}
\mathcal{A}(\rho;t,\hat x)_x&=& \rho(t,\hat x+h)w(h) - \rho(t,\hat x)w(0) + \int_{\hat x}^{\hat x+h} -\rho(t,y) w'(y-x)\; ds,\\
\mathcal{A}(\hat Q;\hat x)_x&=& 
\hat Q(\hat x+h)w(h) - \hat Q(\hat x)w(0) + \int_{\hat x}^{\hat x+h} -\hat Q(y) w'(y-x)\; ds, 
\end{eqnarray*}
and using $w'(x) \le 0$, we get 
\begin{eqnarray}
&& \hspace{-1cm}\mathcal{A}(\rho;t,\hat x)_x-
\mathcal{A}(\hat Q;\hat x)_x 
\nonumber \\
&=&[\rho(t,\hat x+h)-\hat Q(t,\hat x+h)] w(h) 
-\int_{\hat x}^{\hat x+h} [\rho(t,y)-\hat Q(y)] w'(y-x)\; ds ~<~ 0.
\label{eq:s5}
\end{eqnarray}

By using~\eqref{eq:p1}, we compute, at $(t,\hat x)$, 
\begin{eqnarray}
[\rho(t,\hat x) v(\mathcal{A}(\rho;t,\hat x))]_x &=&
[\rho(t,\hat x) v(\mathcal{A}(\rho;t,\hat x))]_x -[\hat Q(\hat x) v(\mathcal{A}(\hat Q;\hat x))]_x 
\nonumber  \\
&=& \hat Q'(\hat x) [v(\mathcal{A}(\rho))-v(\mathcal{A}(\hat Q))] + \hat Q(\hat x) v'(\mathcal{A}(\rho)) [
\mathcal{A}(\rho)_x-\mathcal{A}(\hat Q)_x]  \nonumber
\\
&& +\,
\hat Q(\hat x) \mathcal{A}(\hat Q)_x [v'(\mathcal{A}(\rho))-v'(\mathcal{A}(\hat Q))].
\label{eq:stabest} 
\end{eqnarray}
Since the profile $\hat Q(x)$ is monotone increasing, we have $\mathcal{A}(\hat Q)_x>0$.
By further using the properties~\eqref{eq:s3}--\eqref{eq:s5}, 
all three terms on the righthand side of~\eqref{eq:stabest} are positive.
We conclude~\eqref{eq:s1}.

\medskip

Similarly, let $\tilde Q(x)\;\dot=\;Q(x+\tilde c)$ for some $\tilde c\in\mathbb{R}$ 
be a profile such that
$\tilde Q(x) \le \rho(t,x)$  for all $x$, 
and let $\hat x$ be the largest $x$ value where $\tilde Q(\hat x)=\rho(\hat x)$. 
By a completely similar argument we conclude that  $\rho_t(t,\hat x) >0$. 
The proof is completely similar and we omit the details. 

\medskip

\noindent\textbf{Step 2.}
The time asymptotic stability follows from the properties in Step 1, 
with similar arguments as in Step 2 of the proof of Theorem~\ref{th3b}. 
We skip the details.
\end{proof}

\textbf{Numerical simulations.} 
Solutions for the nonlocal conservation law~\eqref{eq:claw1}  using oscillatory initial condition
\begin{equation}\label{SimIC2}
\rho(0,x) = \begin{cases} 0.2, & x \le -0.3,\\
0.5-0.3*\sin(5\pi x), \qquad& -0.3 < x< 0.3,\\
0.8, & x\ge 0.3.
\end{cases}
\end{equation}
are computed  with a finite difference method,  at various time $t$.  
See Figure~\ref{fig:4}, where two cases of the weight functions $w(\cdot)$ are treated. 
In the plots in the top row we have $w'<0$. 
Here, we observe that the oscillations are quickly damped
and that the solution approaches some profile as $t$ grows.
On the other hand, in the plots in the bottom row we have $w'>0$, 
and we observe that the solution becomes more oscillatory as $t$ grows, 
indicating  instability of the profiles. 
This behavior is analyzed in some detail in Section~\ref{sec:insta}.

\begin{figure}[htbp]
\begin{center}
\includegraphics[width=15cm]{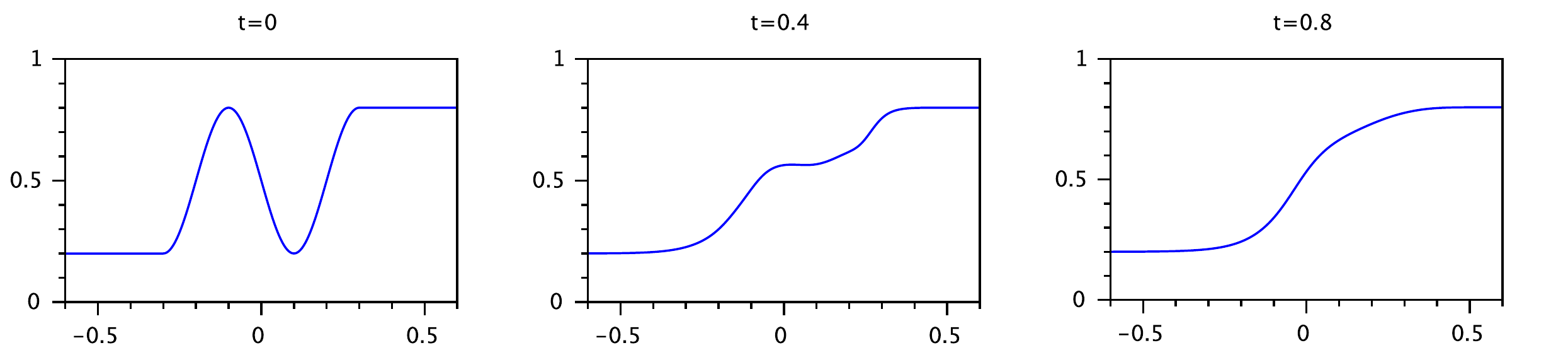}
\includegraphics[width=15cm]{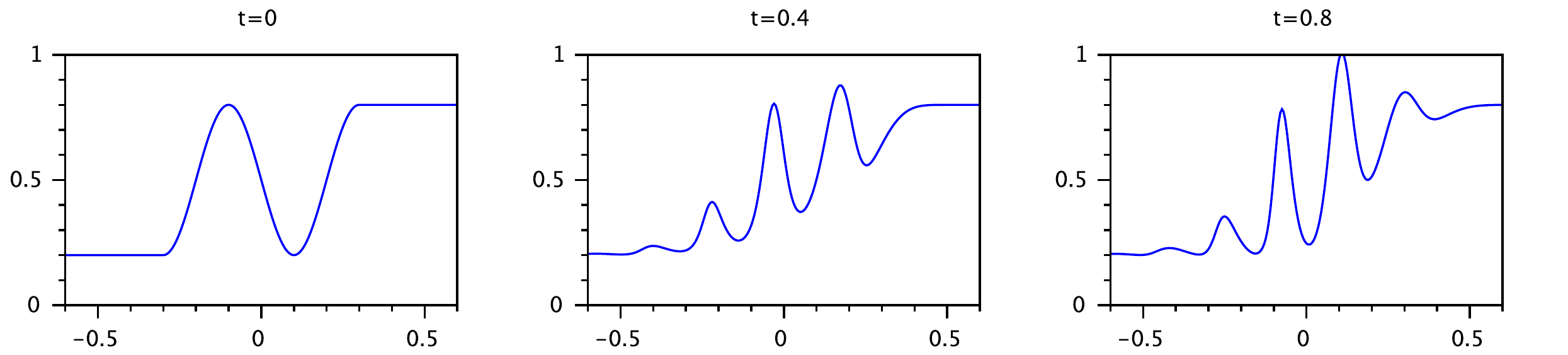}
\caption{Solution for the nonlocal conservation law~\eqref{eq:claw1} 
with oscillatory initial condition, at $t=0, 0.4, 0.8$.
Top: With $w(x)=\frac{2}{h}-\frac{2x}{h^2}$, i.e., $w'<0$,  the solution quickly 
approaches the traveling wave profile $Q(x)$ as $t$ grows.
Bottom: With $w(x)=\frac{2x}{h^2}$, i.e., $w'>0$,
the solution becomes more oscillatory as $t$ grows.}
\label{fig:4}
\end{center}
\end{figure}

\section{Micro-macro limit of traveling waves}
\label{sec:conv}
\setcounter{equation}{0}

\begin{theorem}[Micro-macro limit of traveling waves]
\label{th:convTW}
Fix a weight function $w(\cdot)$ that satisfies the assumption (A1). 
Let $\{\ell_n\}$ be a sequence of car length such that $\lim_{n\to\infty}\ell_n=0$. 
Let $P^{(n)}(\cdot)$ be the discrete stationary profile for the FtLs model,
with car length $\ell_n$, such that
\[
\lim_{x\to\pm\infty} P^{(n)}(x)=\rho^\pm,\qquad P^{(n)}(0)=\hat\rho.
\]
Then, as $n\to\infty$, the sequence of functions 
$P^{(n)}(\cdot)$ converges to a unique limit function
$Q(\cdot)$, where $Q(\cdot)$ is the stationary profile for 
the non-local conservation law~\eqref{eq:claw1},
satisfying the  conditions
\begin{equation}\label{eq:PlimQ}
\lim_{x\to\pm\infty} Q(x)=\rho^\pm,\qquad Q(0)=\hat\rho.
\end{equation}
\end{theorem}

\begin{proof}
\textbf{Step 1.}  
Let $\lambda_+^\ell>0$ be the exponential rate for the profile $P(\cdot)$
as $x\to\infty$, derived in Lemma~\ref{lm:A}, where $\lambda_+^\ell$ is the unique
solution of~\eqref{eq:G22}. 
Let $\lambda_+$ be the exponential rate for the profile $Q(\cdot)$ derived 
in Lemma~\ref{lm:AL},
where $\lambda_+$ is the unique solution of~\eqref{eq:lambda}.
Recalling~\eqref{eq:G22}, we define the function
\begin{equation}\label{eq:H}
H(a,\lambda) 
\;\dot=\;  b \sum_{k=0}^{ m} 
\hat w_{k}   e^{-ka\lambda} 
-\frac{a\lambda}{ 1-e^{-a\lambda}}.
\end{equation}
Recall that $a=\ell/\rho^+$. 
Since the first term on the righthand side of~\eqref{eq:H} 
is an approximate Riemann sum of the integral 
$ \int_0^h e^{-\lambda s}w(s)\,ds$, where $a$ is the length of the sub-intervals,
we have 
\begin{eqnarray*}
H(0,\lambda) &\dot=& \lim_{a\to0}H(a,\lambda) \;= \;b\int_0^h e^{-\lambda s}w(s)\;ds -1 .
\end{eqnarray*}
By~\eqref{eq:G22} and~\eqref{eq:lambda}, we have 
\[
 H(0,\lambda_+)=0, \qquad H(a, \lambda_+^\ell)=0.
\]

For $0<\lambda \le \lambda^\sharp \,\dot=\, \max\{\lambda^\ell_+,\lambda_+\}$, we have 
\[
\left| \frac{\partial}{\partial a} H(a,\lambda) \right| = 
\left| 
- b \sum_{k=0}^{ m} \hat w_k k\lambda e^{-ka\lambda} 
- \frac{\lambda (1-e^{-a\lambda})-a\lambda^2e^{-a\lambda}}{(1-e^{-a\lambda})^2}
\right|
~\le~ M_1,
\]
and
\[
\frac{\partial}{\partial \lambda}  H(0,\lambda) =
\int_0^h -s e^{-\lambda s}w(s)\;ds  \le - e^{-\lambda h} \int_0^h s w(s)\; ds
\le - M_2 <0,
\]
so
\[
\abs{\frac{\partial}{\partial \lambda}  H(0,\lambda) } \ge M_2.
\]

Since $\lambda\mapsto  H(0,\lambda) $ is strictly decreasing, 
we have $H(0,\lambda_+^\ell) \not= 0$. 
Then, it holds
\[
\abs{\lambda_+-\lambda_+^\ell} = 
\abs{\frac{\lambda_+-\lambda_+^\ell}{H(0,\lambda_+)-H(0,\lambda_+^\ell)}}
 \cdot 
\abs{\frac{H(0,\lambda_+^\ell)-H(a,\lambda_+^\ell))}{a} }\cdot
a
\le \frac{\ell}{\rho^+} \cdot \frac{M_1}{M_2},
\]
therefore, we have 
\[
\lim_{\ell\to 0} \lambda_+^\ell = \lambda_+. 
\]
A completely similar proof yields 
\[
\lim_{\ell\to 0} \lambda_-^\ell = \lambda_-. 
\]

\bigskip

\textbf{Step 2.} Fix a small $\ell>0$, and let $P^\ell(\cdot)$ be a profile 
that satisfies~\eqref{eq:P} with asymptotic conditions 
$\lim_{x\to\pm\infty}P^\ell(x)=\rho^\pm$.
Then $P^\ell(\cdot)$ is monotone and Lipschitz continuous. 
Taking the limit $\ell\to 0$, by Helly's compactness theorem 
there exists a subsequence of $\{P^\ell\}$ 
that converges to a limit function $Q(\cdot)$
uniformly on bounded set. 
Thanks to the asymptotic conditions  
$\lim_{x\to\pm\infty} P^{\ell} (x)=\pm\infty$ for all $\ell >0$
and the result in step 1 on the exponential rates, 
the convergence $P^{\ell} \to Q$ is uniform for all $x\in\mathbb{R}$. 
Moreover, the limit function $Q$ is monotone, Lipschitz continuous, and satisfies
\[
Q(0)=\hat\rho, \quad 
\lim_{x\to-\infty} Q(x) = \rho^-, 
\quad 
\lim_{x\to\infty} Q(x) = \rho^+.
\]

It remains to show that $Q$ satisfies the integral equation~\eqref{eq:p1}. 
Indeed, recalling the definition of the operator $A^{P^\ell}$ in \eqref{eq:AP}, 
we write
\[
A^{P^\ell}(x) 
= \sum_{k=0}^\infty \int_{(L^{P^\ell})^k(x)}^{(L^{P^\ell})^{k+1}(x)} 
P^\ell\left((L^{P^\ell})^k(x)\right) w(y-x)\; dy .
\]
Since the convergence $P^\ell \to Q$ is uniform for all $x\in \mathbb{R}$, 
and the weight function $w$ is continuous on its support $x\in[0,h]$, we have
\begin{equation}\label{FF}
\lim_{\ell\to 0+} A^{P^\ell}(x)  =\int_{x}^{x+h} Q(y) w(y-x)\; dy= 
 \mathcal{A}(Q;x) \qquad \forall x\in\mathbb{R}.
\end{equation}

By the periodic property, $P^\ell $ satisfies the integral equation
\[
\int_x^{x+\ell/P^\ell(x)} \frac{1}{v(A^{P^\ell}(z))}\; dz = \frac{\ell}{\bar f} \qquad \forall x\in\mathbb{R}.
\]
This can be written as 
\begin{equation}\label{FF2}
\frac{1}{\ell/P^\ell(x)} \int_x^{x+\ell/P^\ell(x)} \frac{1}{v(A^{P^\ell}(z))}\; dz = \frac{P^\ell(x)}{\bar f} \qquad \forall x\in\mathbb{R},
\end{equation}
where the left-hand side is the average value of $v(A^{P^\ell}(z))^{-1}$ 
over the interval $[x, x+ \ell/P^\ell(x)]$. 

Taking the limit $\ell\to 0+$ on both sides of~\eqref{FF2}, we obtain
\[
\frac{1}{v(\mathcal{A}(Q;x))} = \frac{Q(x)}{\bar f}\qquad \implies \quad 
Q(x) v(\mathcal{A}(Q;x)) = \bar f, \qquad 
\forall x\in\mathbb{R}.
\]
Thus, we conclude that $Q$ satisfies \eqref{eq:p1}, completing the proof. 
\end{proof}

\section{Further Discussions}
\setcounter{equation}{0}
\label{sec:insta}

\subsection{Traveling waves with non-zero speed}
So far in this paper we considered stationary profiles. 
In the case where a traveling wave has non-zero speed, 
a coordinate translate can be used. 
Let $\sigma$ be the velocity of the travelling waves. 
Assume that the function $f(\rho)=\rho v(\rho)$ is convex with $f'' <0$. 
Let $\rho^-\in\mathbb{R},\rho^+\in\mathbb{R}$ satisfy
\[
0<\rho^-<\hat\rho_\sigma < \rho^+<1, \qquad 
\sigma =  \frac{f(\rho^-)-f(\rho^+)}{\rho^--\rho^+},\qquad 
\mbox{where}~ f'(\hat\rho_\sigma ) = \sigma.
\]

We consider the nonlocal conservation law~\eqref{eq:claw1}. 
Let $\xi=x-\sigma t$. We seek profile $Q(\cdot)$ such that 
$\rho(t,x)=Q(\xi)=Q(x-\sigma t)$ 
is a solution for~\eqref{eq:claw1}. 
The profile $Q(\cdot)$ satisfies the ODE
\[
-\sigma Q_\xi + (Q \cdot v(\mathcal{A}(Q;\xi)))_\xi =0,
\]
where $\mathcal{A}$ is the averaging operator defined in~\eqref{eq:AQ}. 
This leads to the integro-equation 
\begin{equation}\label{Q-sigma}
Q(\xi) \left[v(\mathcal{A}(Q;\xi)) - \sigma  \right] \equiv \bar f_\sigma,
\qquad \mbox{where} \quad 
\bar f_\sigma = f(\rho^-) - \sigma \rho^- = f(\rho^+) - \sigma \rho^+.
\end{equation}
The analysis in Section~\ref{sec:NLCL} can be applied to~\eqref{Q-sigma} in a similar way,
achieving similar results.

\medskip

On the other hand, 
for the FtLs model we seek a profile $P(\cdot)$ such that 
\[  P(z_i(t)-\sigma t) = \rho_i(t).\]
Differentiating this in $t$, and after direct computation we arrive at
\begin{equation}\label{P-sigma}
P'(\xi) = \frac{P^2(\xi)}{ \ell \left[ v(A^P(\xi)) -\sigma\right]} \cdot
\Big[  v(A^P(L^P(\xi))) -  v(A^P(\xi))\Big].
\end{equation}
Here $L^P$ is the operator in~\eqref{eq:LP} and $A^P$ is defined in~\eqref{eq:AP}. 
Note the similarity between~\eqref{P-sigma} and~\eqref{eq:P}. 
Again, similar results are achieved by applying the same approach as 
in Section~\ref{sec:NLF}. 

\bigskip

\subsection{Unstable profiles with $w'>0$} 
 
We  discuss a case where the traveling wave profiles  are not local attractors
for the solution of the Cauchy problem for the nonlocal conservation law. 
Assume that the weight function $w(\cdot)$ satisfies 
\[ w'(x) >0, \qquad x\in[0,h]. \]
This implies that, on the interval $[0,h]$ ahead of a driver, 
the situation further ahead is more important than the one closer to the driver.
Of course, from a practical point of view, this assumption is rather obscure, and
we expect  the mathematical models to exhibit erroneous behavior.
Indeed, as we have observed in the numerical simulations (see Figure~\ref{fig:4}), 
when $w'(x)>0$ on $x\in(0,h)$, the solution of the Cauchy problem for~\eqref{eq:claw1}
does not approach any profile $Q(\cdot)$ as $t$ grows. 
Here we offer a supplementary argument. 

To fix ideas, assume that $w(0)=0$ and $w'(h) \ge c_o >0$ on $(0,h)$. 
We revisit the proof of Theorem~\ref{th3} and observe that~\eqref{eq:stabest} and~\eqref{eq:s3} give
\begin{eqnarray*}
\rho_t(t,\hat x) &=& - [\rho(t,\hat x) v(\cA(\rho;t,\hat x))]_x \\
&=& - Q'(\hat x) [v(\mathcal{A}(\rho))-v(\mathcal{A}(Q))] - Q v'(\mathcal{A}(\rho)) [
\mathcal{A}(\rho)_x-\mathcal{A}(Q)_x]\\
&& \qquad  -
Q \mathcal{A}(Q)_x [v'(\mathcal{A}(\rho))-v'(\mathcal{A}(Q))] \\
&=& \int_0^h \left[ Q'(\hat x) v'(p_1) w(s) + Q(\hat x) A(Q;\hat x)_x v''(p_2) w(s) -Q(\hat x) v'(\cA(\rho;t,\hat x)) w'(s)\right] \\
&& \qquad 
\cdot \left[ Q(\hat x+s) - \rho(t,\hat x+s) \right] \, ds,
\end{eqnarray*}
where
\[
v'(p_1) \;\dot=\; \frac{v(\cA(\rho))-v(\cA(Q))}{\cA(\rho)-\cA(Q)},
\qquad 
v''(p_2)  \;\dot=\; \frac{v'(\cA(\rho))-v'(\cA(Q))}{\cA(\rho)-\cA(Q)}.
\]
Since the profile $Q(x)$ approaches $\rho^+$ as $x\to +\infty$, therefore, for 
$\hat x$ sufficiently large, $Q'(\hat x)$ and $A(Q;\hat x)_x$ become arbitrarily small, 
and we have 
\[ \rho_t(t,\hat x) >0 .\]
This shows that the solution $\rho(t,x)$ can never settle around $Q(x)$ for $x$ large,
therefore the profiles $Q$ are not time asymptotic limits
in the sense of Theorem~\ref{th3}.

\subsection{A symmetric kernel $w$}

We now consider the case when
a driver consider the situation both in front and behind the car. 
Although in reality one only adjusts  the speed according to the leaders, 
there has been an interest in nonlocal models where the 
integral kernel has support both in front and behind particle.
This leads to the non-local conservation law
\begin{equation}\label{eq:clawB}
\rho_t + \left[ \rho(t,x) \cdot v\left(\int_{x-{h_1}}^{x+{h_2}} \rho(t,y) w(y-x) \; dy\right)\right]_x =0,
\end{equation}
where the weight function $w(\cdot)$ has support on $[-h_1,h_2]$.

One motivation for~\eqref{eq:clawB} stems from possible extensions of 
the one-dimensional conservation law into  several space
dimensions, for applications such as pedestrian flow and flock flow. 
A radially symmetric kernel is commonly used, i.e.~$w(\vec{x})=w(r)$ where $r=|\vec{x}|$. 
Then, the corresponding one-dimensional kernel $w(x)$ is necessarily 
an even function. 
We now consider~\eqref{eq:clawB} with $h_1=h_2=h$ and 
a weight function $w(\cdot)$ that satisfies
\begin{equation}\label{eq:ce1}
w(s) =0  \quad \forall |s| \ge h, 
\qquad  w(-s)=w(s) \quad \forall s, \qquad \int_0^h w(s)\, ds=0.5.
\end{equation}

We seek stationary smooth and monotone profiles $\tilde Q(\cdot)$  such that
\begin{equation}\label{eq:ce2}
\lim_{x\to\pm\infty} \tilde Q(x) = \rho^\pm, \qquad f(\rho^-)=f(\rho^+)=\bar f.
\end{equation}
The profile $\tilde Q(\cdot)$ satisfies the integral equation
\begin{equation}\label{eq:ce3}
\tilde Q(x) \cdot v\left(\int_{x-h}^{x+h} \tilde Q(y) w(y-x) \; dy\right) \equiv \bar f .
\end{equation}
Carrying out a similar asymptotic analysis as in Lemma~\ref{lm:AL}, 
for the limit $x\to\infty$, 
equation~\eqref{eq:lambda} is modified to 
\begin{equation}\label{eq:lambdaN}
\int_{-h}^h e^{-\lambda s} w(s) \; ds = \frac{1}{\beta}.
\end{equation}
We seek positive roots  $\lambda>0$. By using $w(-s)=w(s)$, we get
\begin{eqnarray*}
\frac{1}{\beta} &=& \int_{-h}^0 e^{-\lambda s} w(s) \; ds + \int_{0}^h e^{-\lambda s} w(s) \; ds
~=~ \int_0^h \left[ e^{-\lambda s} + e^{\lambda s} \right] w(s)\; ds \\
&=& \int_0^h  2 \cosh(\lambda s) w(s)\; ds  
~>~ 2 \int_0^h   w(s)\; ds 
~=~ 1.
\end{eqnarray*}
Therefore, a positive root $\lambda_+$ exists if and only if $\beta<1$, i.e., $\rho^+ < \hat \rho$.
Thus, as $x\to +\infty$, the profile 
$\tilde Q$ converges to $\rho^+$
exponentially if and only if 
$\rho^+ < \hat \rho$. 
A completely similar argument shows that 
as $x\to -\infty$, the profile 
$\tilde Q$ converges to $\rho^-$
exponentially if and only if 
$\rho^- > \hat \rho$. 

The existence of solutions for~\eqref{eq:ce3} is not obvious, since the 
technique used in Section~\ref{sec:NLCL}
 can not be applied. 
However, if monotone and continuous profiles  should exist, 
then the above argument indicates that 
$\rho^+<\hat \rho<\rho^-$ (instead of  $\rho^+>\hat \rho>\rho^-$ as in 
Lemma~\ref{lm:AL}). 
In the limit as $\ell \to 0+$, these profiles converge 
to a downward jump with $\rho^- > \rho^+$.

\section{Another non-local model: averaging the velocity}
\label{sec:model2}
\setcounter{equation}{0}

In this section we consider the conservation law~\eqref{eq:claw2}
and the corresponding FtLs model~\eqref{FtLs2}. 
Note that~\eqref{FtLs2} now  gives
\begin{align}\label{rhodot2}
    \dot \rho_i(t) 
=-\frac{\ell \left( \dot z_{i+1} - \dot z_i \right) }{(z_{i+1}-z_i)^2} 
= \frac{1}{\ell} \rho_i^2 \cdot \big( v_i^*(\rho;t)- v^*_{i+1}(\rho;t)\big),
\qquad i\in\mathbb{Z}.
\end{align}

The same results on stationary profiles as in Sections~\ref{sec:NLCL} and~\ref{sec:conv} 
apply to these models.
The proofs are very similar, with only mild modifications. 
Below we go through the analysis briefly,  focusing mainly on  the differences. 

\subsection{The FtLs model}

We seek  ``discrete stationary traveling wave profiles'' $\cP(\cdot)$ such that 
\begin{equation}\label{cP1}
\cP(z_i(t)) = \rho_i(t), \qquad \forall t\ge 0, \quad \forall i\in\mathbb{Z}.
\end{equation}
We define the  operators
\begin{equation}\label{eq:LcP}
L^\cP(x) \;\dot=\; x + \frac{\ell}{\cP(x)}, \qquad
A(v(\cP(z_i))) \;\dot=\; \sum_{k=0}^m w_{i,k} v(\cP((L^\cP)^k(z_i) )) .
\end{equation}

After a similar derivation, we find that $\cP(x)$ satisfies 
a delay differential equation,
\begin{equation}\label{eq:cP}
\cP'(x) = - \frac{\cP^2(x)}{\ell \cdot A(v(\cP(x))) } \Big[ A(v(\cP(L^\cP(x)))) - A(v(\cP(x)))\Big].
\end{equation}

The asymptotic limits at $x\to\pm\infty$, 
the periodic behavior, the existence and uniqueness of solutions of the initial value problems,  
and the existence and uniqueness of two-point asymptotic value problem 
all follow in almost the same way as those in Section~\ref{sec:NLCL}. 

\medskip

\textbf{Stability.} The analysis for the stability of the profiles is slightly different. 
In the same setting as in the proof of Theorem~\ref{th3b}, we let $k$ be the index such that
\[
\hat \cP(z_k(t)) =\rho_k(t), \quad \hat \cP(z_i(t)) > \rho_i(t) \quad \forall i >k,
\qquad \mbox{and}\quad
\hat \cP(z_i(t)) \ge \rho_i(t) \quad \forall i\in\mathbb{Z} ,
\]
and claim that 
\begin{equation}\label{eq:ww2}
\frac{\dot\rho_k}{\dot z_k}  < \hat \cP'(z_k). 
\end{equation}

Indeed, using $L^{\hat \cP}(z_k)=z_{k+1}$, we have
\begin{eqnarray*}
\hat \cP'(z_k) &=&
 \frac{\hat \cP(z_k)}{A(v(\hat \cP(z_k))} \cdot 
\frac{A(v(\hat \cP(z_{k+1})))-A(v(\hat \cP(z_k)))}{v(\hat \cP(z_{k+1})) -v(\rho_k) }
\cdot
\frac{v(\hat \cP(z_{k+1})) -v(\rho_k) }{z_k - z_{k+1}} 
\;\dot=\; A_1 B_1 C_1, \\
\frac{\dot\rho_k}{\dot z_k} &=& \frac{\rho_k}{v^*_k} \cdot 
\frac{v^*_{k+1}-v^*_k}{ v(\rho_{k+1}) - v(\rho_k)} \cdot
\frac{ v(\rho_{k+1}) - v(\rho_k)}{z_k- z_{k+1}} 
\;\dot=\; A_2 B_2 C_2.
\end{eqnarray*}
Then, using $v'<0$, we have
\begin{eqnarray*} 
A(v(\hat \cP(z_k))) < v^*_k \quad &\implies& \quad A_2 < A_1,  \\
\hat \cP(L^{\hat \cP}(z_k)) > \rho_{k+1} \quad &\implies& \quad
v(\hat \cP(L^{\hat \cP}(z_k))) < v(\rho_{k+1} ) \quad \implies \quad   C_2 < C_1.
\end{eqnarray*}

It remains to show that $B_2\le B_1$.  Indeed, we observe that 
\[
v(\rho_{k+1}) - v(\rho_k)  > v(\hat \cP(z_{k+1})) -v(\rho_k)  .
\]
Now  
let $V_1(\cdot), V_2(\cdot) $ denote the piecewise constant functions defined as
\begin{eqnarray*}
V_1(x) \;\dot=\; v(\hat \cP((L^{\hat \cP})^j(z_k))), \qquad &&\mbox{for}\quad 
(L^{\hat \cP})^j(z_k) ~<~ x ~<~ (L^{\hat \cP})^{j+1}(z_k), \\
V_2(x) \;\dot=\; \rho_{k+j}, \qquad &&\mbox{for}\quad 
z_{k+j} ~<~ x ~<~ z_{k+j+1}.
\end{eqnarray*}
In this  setting we have 
\[
V_1(x)=V_2(x), \quad   x\in(z_k,z_{k+1}); \qquad 
\mbox{and}\quad 
V_1(x) < V_2(x) ,  \quad  x > z_{k+1}.
\]
Now, we can write
\begin{eqnarray*}
&& \hspace{-2.6cm} A(v(\hat \cP(z_{k+1})))-A(v(\hat \cP(z_k))) \\
&=& -\rho_k \int_{z_k}^{z_{k+1}} w(y-z_k)\; dy
+ \int_{z_{k+1}}^\infty (w(y-z_{k+1})-w(y-z_k)) V_1(y)\; dy,\\
v^*_{k+1}-v^*_k &=& -\rho_k \int_{z_k}^{z_{k+1}} w(y-z_k)\; dy + 
\int_{z_{k+1}}^\infty (w(y-z_{k+1})-w(y-z_k)) V_2(y)\; dy.
\end{eqnarray*}
Now, since $w'\le 0$, we have $w(y-z_{k+1})-w(y-z_k)\ge 0$.
We  conclude
\[
v^*_{k+1}-v^*_k \le  A(v(\hat \cP(z_{k+1})))-A(v(\hat \cP(z_k))). 
\]
This implies $B_2\le B_1$, and therefore  proves~\eqref{eq:ww2}. 

\begin{remark}
Note that in the above proof  we do not use the assumption $v''\le 0$. 
\end{remark}

\subsection{The nonlocal conservation law}

Let $\cQ(x)$ denote  the stationary wave profile  for~\eqref{eq:claw2}.
Introduce the operator $\mathcal{A}$ such that 
\begin{eqnarray*}
\mathcal{A}(v(\rho);t,x) 
&\dot=&
\int_x^{x+h} v(\rho(t,y)) w(y-x) \; dy~ =~ \int_0^h v(\rho(t,x+s)) w(s)\; ds,\\
\mathcal{A}(v(\cQ);x) 
&\dot=& \int_x^{x+h} v(\cQ(y)) \,w(y-x) \; dy
~=~\int_0^h v(\cQ(x+s)) w(s) \; ds.
\end{eqnarray*}
A stationary solution $\cQ(\cdot)$ for~\eqref{eq:claw2}
 satisfies the equation
\begin{equation}\label{eq:cQ1}
\cQ(x) \cdot \mathcal{A}(v(\cQ);x) 
\equiv \bar f = \mbox{constant}= f(\rho^\pm).
\end{equation}
This can also be written in the form of a delay integro-differential equation,
\[
\cQ'(x) = -\frac{\cQ(x)}{\mathcal{A}(v(\cQ);x) } \cdot
\int_0^h v'(\cQ(x+s)) \cQ'(x+s) w(s) \; ds.
\]

The asymptotic limits are analyzed in the same way as for Section~\ref{sec:NLCL}. 

\medskip

\textbf{Approximate solutions for initial value problem.}
The approximate solutions for the initial value problem are generated using  a similar 
algorithm  as in Step 1 of the proof for Theorem~\ref{th1}, with a few different details. 
Fix a $k\in\mathbb{Z}$, let 
 $\cQ_i$ for $i\ge k$ be given.  We compute $\cQ_{k-1}$ by 
solving the nonlinear equation
\[
G(\cQ_{k-1}) \;\dot=\; \cQ_{k-1} \cA(v(\cQ); x_{k-1})- \bar f =0, 
\]
where on $x\in[x_{k-1},x_k]$ the function $\cQ(x)$ is reconstructed by linear interpolation. 
The above nonlinear equation has a unique zero, if it is monotone. 
We claim that, for $\dx$ sufficiently small,  
\[
\partial_{\cQ_{k-1}} G(\cQ) = \cA(v(\cQ); x_{k-1}) + \cQ_{k-1} \partial_{\cQ_{k-1}} \cA(v(\cQ); x_{k-1})>0.
\]
Indeed, we have
\begin{eqnarray*}
\partial_{\cQ_{k-1}} \cA(v(\cQ); x_{k-1})
&=& \int_{x_{k-1}}^{x_k} \partial_{\cQ_{k-1}}
v\left(\cQ_{k-1}\frac{x_k - x}{x_k-x_{k-1}}    +\cQ_k \frac{x-x_{k-1}}{x_k-x_{k-1}} \right) 
w(y-x)\; dy\\
&=&  -\O(1) \cdot  \int_{x_{k-1}}^{x_k}
 \frac{x_k - x}{x_k-x_{k-1}}  w(y-x)\; dy  ~=~ -\O(1) \cdot \Delta x,
\end{eqnarray*}
proving the claim. 

The existence and uniqueness of the initial value problem and the two-point asymptotic value problem
follow in a very similar way as those in Section~4. 

\medskip

\textbf{Stability.} 
The existence and uniqueness of entropy weak solutions for the Cauchy problem 
of~\eqref{eq:claw2} is recently established in~\cite{FKG2018}. 
Fix a time $t\ge 0$. Similar to the proof of Theorem~\ref{th3}, 
let $\hat \cQ(x)$ be a profile such that $\hat \cQ(x) \ge \rho(t,x)$,
and let $\hat x$ be a point satisfying
\begin{equation}\label{eq:sQ}
\hat \cQ(\hat x) = \rho(t,\hat x), 
\qquad
\hat \cQ_x(\hat x) = \rho_x(t,\hat x), 
\qquad\hat  \cQ(x) > \rho(t,x) \qquad \forall x>\hat x.
\end{equation}

We claim that,
\begin{equation}\label{eq:s12}
\rho_t(t,\hat x) <0, 
\qquad \mbox{i.e.}\quad 
[\rho(t,\hat x) \cA(v(\rho); t,\hat x)]_x >0. 
\end{equation}

Indeed, we compute
\begin{eqnarray}
&&\hspace{-1cm} [\rho(t,\hat x) \cA(v(\rho); t,\hat x)]_x ~=~
[\rho(t,\hat x) \cA(v(\rho); t,\hat x)]_x - \left[\hat \cQ(\hat x) \cA(v(\hat \cQ);\hat x)\right]_x \nonumber \\
&=& \hat \cQ_x \left[ \cA(v(\rho); t,\hat x) -\cA(v(\hat \cQ);\hat x) \right] +
\hat \cQ(\hat x) \left[ \cA(v(\rho); t,\hat x)_x - \cA(v(\hat \cQ);\hat x)_x \right].
\label{eq:EE}
\end{eqnarray}

Since $\hat \cQ(x)$ is monotone increasing, we have $\hat \cQ_x >0$. And, 
by~\eqref{eq:sQ} we have
\[
v(\rho; t,\hat x) > v(\hat \cQ;\hat x) \quad (x>\hat x), \qquad \mbox{so}\quad
\cA(v(\rho); t,\hat x) >\cA(v(\hat \cQ);\hat x).
\]
Thus, the first term on the righthand side of~\eqref{eq:EE} is positive. 
To estimate the second term,  using 
\begin{eqnarray*}
\cA(v(\rho); t,\hat x)_x &=& 
v(\rho(t,\hat x+h))w(h) - v(\rho(t,\hat x))w(0) -\int_0^h v(\rho(t,\hat x+s)) w'(s)\; ds, \\
\cA(v(\hat \cQ);\hat x)_x &=& v(\hat \cQ(\hat x+h))w(h) - v(\hat \cQ(\hat x))w(0) -\int_0^h v(\hat \cQ(\hat x+s)) w'(s)\; ds ,
\end{eqnarray*}
and that $w'\le 0, v' <0$, we get
\[
\cA(v(\rho); t,\hat x)_x - \cA(v(\hat \cQ);\hat x)_x  >0,
\]
proving the claim.  

\begin{remark}
Note again that we do not need the assumption $v''\le 0$ in the above proof.
\end{remark}

Finally,  as $\ell\to 0$,
the profiles $\cP^\ell(\cdot)$ converges to $\cQ(\cdot)$, following the same argument as in the proof of 
Theorem~\ref{th:convTW}. We omit the details.

\section{Concluding remarks}
\label{sec:cr}
\setcounter{equation}{0}

In this paper we analyze existence, uniqueness and stability of stationary traveling wave profiles for several non-local 
models for traffic flow, for both particle models and PDE models. 
Furthermore, we prove the convergence of the traveling waves of the FtLs models
to those of the corresponding non-local conservation laws. 
However, 
the convergence of solutions of the non-local microscopic model to the macroscopic model
remains open.
We recall that, for the local models, the micro-macro limits are well treated in the literature, 
see~\cite{MR3217759, MR3356989, HoldenRisebro, HoldenRisebro2}.
Existence of solutions for the Cauchy problem of the non-local conservation laws 
is also well studied, cf.~\cite{BG2016}.
We speculate that an adaptation of the approach in~\cite{HoldenRisebro}
combined with the results in~\cite{BG2016}
could yield the micro-macro limit. Details may come in a future work. 

It is also interesting to study stationary profiles for the case where the road condition
is discontinuous, for example where the speed limit has a jump at $x=0$. 
Preliminary results in~\cite{ShenTR} show that the profiles for the models~\eqref{eq:claw1}
and~\eqref{eq:claw2}  are very different. 
In both cases, some profiles are non-monotone, some are non-unique, and some are also unstable, (similar to the results in~\cite{ShenDDDE2017}), 
portraying a much more complex picture.

\bigskip

Codes for the numerical simulations used in this paper can be found: 

\verb+http://www.personal.psu.edu/wxs27/TrafficNL/+

\bigskip

\paragraph{Acknowledgement.} We thank the anonymous reviewers for their 
careful readings  and useful comments which led to an improvement of our manuscript.

\end{document}